\documentclass[a4wide]{amsart}
\usepackage{amsmath,amsfonts,amssymb,amsthm, esint}
\usepackage[english]{babel}
\usepackage{enumerate}
\usepackage{dsfont}

\usepackage[colorlinks,citecolor=blue]{hyperref}

\newtheorem{theorem}{Theorem}[section]

\newtheorem{lemma}[theorem]{Lemma}
\newtheorem{proposition}[theorem]{Proposition}
\newtheorem{corollary}[theorem]{Corollary}

\theoremstyle{remark}
\newtheorem{remark}[theorem]{Remark}

\numberwithin{equation}{section}
\numberwithin{figure}{section}

\newcommand{\R}{\mathbb{R}}
\newcommand{\N}{\mathbb{N}}
\newcommand{\Z}{\mathbb{Z}}
\newcommand{\bra}[1]{\left(#1\right)}
\newcommand{\cur}[1]{\left\{#1\right\}}

\newcommand{\ang}[1]{\left<#1\right>}

\newcommand{\abs}[1]{\left\lvert#1\right\rvert}
\newcommand{\norm}[1]{\left\lVert#1\right\rVert}
\let\d\undefined
\newcommand{\d}{\mathop{}\!\mathrm{d}}

\newcommand{\A}{\mathcal{A}}

\newcommand{\simp}{\operatorname{Simp}}
\newcommand{\chain}{\operatorname{Chain}}

\newcommand{\strat}{\mathsf{strat}}
\newcommand{\ito}{\mathsf{ito}}

\newcommand{\germ}{\operatorname{Germ}}

\newcommand{\dya}{\operatorname{dya}}

\renewcommand{\d}{\mathsf{d}}

\newcommand{\cut}{\operatorname{cut}}
\renewcommand{\fill}{\operatorname{fill}}

\renewcommand{\O}{D}
\newcommand{\conv}{\operatorname{conv}}
\newcommand{\diam}{\operatorname{diam}}

\usepackage{pgf,tikz}
\usetikzlibrary{arrows}

\usepackage{caption,subcaption}
\begin{document}

\author{Giovanni Alberti}
\address{Giovanni Alberti, Dipartimento di Matematica, Universit\`a di Pisa \\
Largo Bruno Pontecorvo 5 \\ I-56127, Pisa}
\email{giovanni.alberti@unipi.it}
\date{\today}

\author{Eugene Stepanov}
\address{St.Petersburg Branch of the Steklov Mathematical Institute of the Russian Academy of Sciences,
Fontanka 27,
191023 St.Petersburg, Russia
%\and
%Department of Mathematical Physics, Faculty of Mathematics and Mechanics,
%St. Petersburg State University, Universitetskij pr.~28, Old Peterhof,
%198504 St.Petersburg, Russia%, email: stepanov.eugene@gmail.com
%\and ITMO University
\and
Faculty of Mathematics, Higher School of Economics, Moscow
}
\email{stepanov.eugene@gmail.com}

\author{Dario Trevisan}
\address{Dario Trevisan, Dipartimento di Matematica, Universit\`a di Pisa \\
Largo Bruno Pontecorvo 5 \\ I-56127, Pisa}
\email{dario.trevisan@unipi.it}

\thanks{This work has been partially supported by the University of Pisa, Project PRA 2018-49 and Gnampa project 2019 ``Propriet\`a analitiche e geometriche di campi aleatori''.
The work of the second author has been also partially financed by
the Program of the Presidium of the Russian
Academy of Sciences \#01 'Fundamental Mathematics and its Applications'
under grant PRAS-18-01 and
RFBR grant \#20-01-00630.
}
\subjclass[2010]{Primary 53C65. Secondary 49Q15, 60H05.}
\keywords{Young integral, It\^{o} integral, Stratonovich integral, Rough paths, sewing lemma.}
\date{\today}

\title[Integration of nonsmooth $\boldsymbol{2}$-forms]{Integration of nonsmooth $\boldsymbol{2}$-forms:
from Young to   It\^{o} and Stratonovich
}

\begin{abstract}
We show that geometric integrals of the type $\int_\Omega f\d g^1\wedge \d g^2$
can be defined over a two-dimensional domain $\Omega$ when the functions
$f$, $g^1$, $g^2\colon \R^2\to \R$ are  just H\"{o}lder continuous
with sufficiently large H\"{o}lder exponents and the boundary of $\Omega$ has sufficiently small dimension, by summing over a refining sequence of partitions the discrete Stratonovich or It\^{o} type terms. This leads to a two-dimensional extension of the classical Young integral that coincides with the integral introduced recently by R.~Z\"{u}st. We further show that the Stratonovich-type summation allows to weaken the requirements on   H\"{o}lder exponents of the map $g=(g^1,g^2)$ when $f(x)=F(x,g(x))$ with $F$ sufficiently regular. The technique relies upon an extension of the sewing lemma from Rough paths theory to
alternating functions of two-dimensional oriented simplices, also proven in the paper.
\end{abstract}

\maketitle

\section{Introduction}
The scope of the present paper is constructing explicitly, via the appropriate discrete approximations, the extension of the classical notion of the integral of the differential $2$-form $f\d g^1\wedge \d g^2$ over any sufficiently nice oriented planar domain $\Omega\subset \R^2$ (one might think for simplicity of $\Omega$ being just
an oriented polygon, or even simpler, a triangle) to the case when the maps $f\colon \R^2\to \R$, $g:=(g_1, g_2)\colon \R^2\to \R^2$
are only H\"{o}lder continuous, so that one might only put the word ``differential'' above in quotation marks, because $g$ might have no derivatives. If $g$ is sufficiently smooth and $f$ just continuous, then $f\d g^1\wedge \d g^2$ can be understood in the modern differential geometry language as 
$f g^*(\d x^1\wedge \d x^2)$, where $\d x^i$ are coordinate $1$-forms, $i=1,2$, and $g^*$ stands 
for the pull-back via $g$, or, alternatively, in a more analytic language,
\begin{equation}\label{eq_fidg1dg2_def0}
\int_\Omega f\d g^1\wedge \d g^2 := \int_\Omega f(x) \det \bra{ \begin{array}{ll} \partial_1 g^1(x) & \partial_2 g^1(x) \\
	\partial_1 g^2 (x) & \partial_2 g^2(x) \end{array}}\, dx,
\end{equation}
$\partial_i$ standing for partial derivatives in the coordinate direction $x_i$, $i=1,2$.
The latter integral is the natural building block for integrals of classical (smooth) differential $2$-forms over smooth parameterized $2$-dimensional surfaces in $\R^n$ via pull-back. 
One comes therefore inevitably to the problem posed when trying to integrate even a very smooth 
differential $2$-form $\omega$ in $\R^n$ over a parameterized H\"older surface $\varphi\colon\Omega\subset\R^2\to \R^n$,  $\varphi(x) = (\varphi^i(x))_{i=1}^n$, letting formally
\[ \int_{\varphi(\Omega)} \omega := \int_{\Omega} \varphi^*\omega,\]
where $\varphi^* \omega$ stands for pull-back of $\omega$ via $\varphi$, i.e.\ 
$\varphi^* \omega := \sum_{i,j} (a_{ij}\circ \varphi) \d \varphi^i \wedge \d \varphi^j$
when $\omega = \sum_{i,j} a_{ij} \d x^i \wedge \d x^j$.

\subsection{History} 
\subsubsection{One-dimensional integrals} 
The one-dimensional prototype of this problem, that is, extending the integral of 
a differential $1$-form $u\d v$ over an interval $[a,b]$ of the real axis
to the maps $u, v\colon \R\to \R$
that are only H\"{o}lder continuous, has been solved by
L.C.~Young~\cite{young_inequality_1936} and independently by V.~Kondurar~\cite{kondurar_1937}. They 
defined the respective integral $\int_a^b u\d v$ as a limit in $k$ of a converging sequence of Riemann sums of the type
$\sum_{i=0}^{k-1} u(a_i)(v(a_{i+1})-v(a_{i}))$ over an appropriate sequence of refining partitions of the interval
$[a,b]$ by consequtive points $a_0:=a < a_1<\ldots < a_k:=b$, thus
mimicking the definition of the classical Riemann integral. This provides
an extension of the latter to the case $u\in C^\alpha(\R)$, $v\in C^\beta(\R)$ when $\alpha+\beta>1$
(later several generalizations of this result for wider classes of functions were provided, 
see e.g.~\cite{young38-stieltjes} as well as the recent paper~\cite{Yaskov19-stieltjes} and references therein).
It is worth remarking that 
the original proof of Young~\cite{young_inequality_1936} was quite 
``handmade'', just by the repetitive use of H\"{o}lder inequality. 
Rather, nowadays  
it is a custom to do it in a more ``automated'' way by using the so-called one-dimensional \textit{sewing lemma}~\cite[lemma 2.1]{feyel_curvilinear_2006}, which together with the construction of this integral, 
now usually called \textit{Young integral}, is one of the basic pillars of the modern theory of \textit{Rough paths}~\cite{friz_course_2014,gubinelli_controlling_2004}~\footnote{A historic curiosity: the modern construction of the
Young integral via sewing lemma is closer to the original one used by Kondurar in~\cite{kondurar_1937} although his contribution to the subject seems to be unfortunately not so well-known.}.

Note that 
in the summands
$u(a_i)(v(a_{i+1})-v(a_{i}))$ one could 
replace $u(a_i)$ by, for instance,
\[
\bar u_{[a_ia_{i+1}]}:= \frac 1 2 (u(a_i)+u(a_{i+1})),
\]
thus leading to a different notion of integral.
 Minding the obvious analogy with stochastic \textit{It\^{o}} (resp.\ \textit{Stratonovich}) integration, we will further call these two constructions It\^{o} (resp.\ Stratonovich) summation. The general conditions on functions $u$ and $v$ for the limits in each of these cases to exist have been studied in~\cite{Smith25-stieltjes} (in the subsequent paper~\cite{WrightBaker69-stieltjes} even 
more general weighted averages of $u$ 
in place of $\bar u_{[a_ia_{i+1}]}$ were considered). Finally, V.~Matsaev and M.~Solomyak constructed in~\cite{MacaevSolomjak72} a similar integral substituting 
$\bar u_{[a_ia_{i+1}]}$ by
the integral average $\fint_{[a_i, a_{i+1}]} u$, which extends the classical integral of 
a smooth differential $1$-form $u\d v$ over an interval to the case when $v\in C^\beta(\R)$
is H\"{o}lder continuous and $u$ belongs to the Besov space $B_{1,1}^\alpha$ with $\alpha+\beta\geq 1$.
In all the mentioned cases the result is the same for $u\in C^\alpha(\R)$, $v\in C^\beta(\R)$ with $\alpha+\beta>1$, but may be different for more general functions. 

\subsubsection{Multidimensional integrals} 
Subsequently, several ways were proposed to extend the above mentioned one-dimensional constructions to multidimensional cases,
notably~\cite{towghi_multidimensional_2002,chouk_2014}, which however 
lack the very important geometric property of the classical integral of multidimensional forms, namely, that of being \textit{alternating}, i.e.\  changing sign with the change of domain orientation (although we also have to mention quite a different and purely geometric approach 
of~\cite{harrison_differential_2006} allowing to treat integration of smooth 
differential forms over nonsmooth domains, e.g.\ having fractal boundary, 
and a quite curious recent construction of~\cite{Yam08}, 
reducing the multidimensional integral to a one-dimensional one involving a Peano-like curve). 

A different approach to the definition of a multidimensional 
integral of nonsmooth ``differential forms'' has been taken by 
R.~Z\"{u}st~\cite{zust_integration_2011}. Applied to the 2D situation which is
of interest in the present paper, it shows that
the integral~\eqref{eq_fidg1dg2_def0}
defined over smooth maps, 
admits the unique extension by continuity with respect to the natural topology of pointwise convergence with bounded H\"{o}lder constants to a multilinear continuous functional
\[
(f, g^1, g^2)\in  C^{\alpha}(\R^2) \times C^{\beta_1}(\R^2) \times
C^{\beta_2}(\R^2)\mapsto I(f, g^1, g^2)
\]
 vanishing over degenerate rectangles and triangles (namely, those having zero area) and alternating in the last two entries, if
$\alpha+ \beta_1 + \beta_2 >2$.
 This functional can be therefore naturally called an integral
\[ 
\int_\Omega f\d g^1\wedge \d g^2 := I(f, g^1, g^2), 
\]
and can be approximated by sums over triangles %$\Delta_i$
forming the sufficiently fine dyadic decomposition of $\Omega$ of 
the functions of three variables (which can be better thought as functions of a triangle) $(p,q,r)\in (\R^2)^3\mapsto \eta_{pqr}$ defined by
%a triangle $[pqr]$ %``germs of differential forms''
\begin{equation}\label{eq_germzust1}
\begin{aligned}
 \eta_{pqr} 
%\eta(p,q,r)
&:= f_{p}
%f(p) 
\int_{\partial [pqr]} g^1\d g^2,
\end{aligned}
\end{equation}
where $f_p:=f(p)$,
the integral above being intended in the sense of Young
(note that in~\cite{zust_integration_2011} a slightly different language 
was used with rectangles instead of triangles; the current 
language is taken from~\cite{SteTrev18a} where a unified approach 
for integration of multidimensional nonsmooth ``differential forms'' called
\emph{``rough differential forms''} up to dimensions $1$ and $2$ was suggested).
R.~Z\"{u}st himself has further successfully employed this integral in several remarkable geometric problems in~\cite{Zust_tree15}.

It is easy to observe that the definition of the integral of $f\d g^1\wedge \d g^2$ through the limit of sums of terms~\eqref{eq_germzust1} over sequences of
refining partitions, is a clear generalization of the construction of the 
one-dimensional Young integral described above. It is inherently based upon
integration by parts, i.e.\ is made so that the Stokes theorem
\[
\int_\Omega \d g^1\wedge \d g^2 =\int_{\partial\Omega} g^1\d g^2
\]
almost automatically be satisfied for appropriate $\Omega\subset \R^2$ 
(rectangle in~\cite{zust_integration_2011} or triangle in~\cite{SteTrev18a}). 
This is however not how one usually expects the integral to be defined: in fact, the Young integrals over the sides of the triangle $[pqr]$ in~\eqref{eq_germzust1} have themselves to be defined either indirectly as continuous extensions of integrals of smooth
differential forms approximating the ``rough differential form'' $g^1\d g^2$ or as a limit of
sums of appropriate discrete approximations (on the contrary, 
the abstract extension of \eqref{eq_fidg1dg2_def0} from spaces of smooth functions to Sobolev or Besov spaces can be done via the technique from \cite{SickelYoussfi99,BrezisNguyen11} dealing with weak Jacobians).

\subsection{Our contribution}
It seems therefore more natural to define the integral of the ``rough differential forms'' $f\, \d g^1\wedge \d g^2$ by \textit{purely discrete} approximations.
To this aim for $f \in C^{\alpha}(\R^2)$, $g^i \in C^{\beta_i}(\R^2)$, $i=1,2$, with $\alpha+ \beta_1 + \beta_2 >2$,  we write
\begin{equation}\label{eq_germzust2}
\begin{aligned}
 \strat_{pqr} &:= \frac 12 \bra{\frac{f_{p} +f_{q} +f_r}{3}} \det \bra{ \begin{array}{ll} \delta g^1_{p q} & \delta g^1_{pr} \\
\delta g^2_{pq} & \delta g^2_{pr} \end{array}},\\
 \ito_{pqr} &:= \frac 12 f_{p} \det \bra{ \begin{array}{ll} \delta g^1_{p q} & \delta g^1_{pr} \\
\delta g^2_{pq} & \delta g^2_{pr} \end{array}} \quad \text{for $[pqr]\subset \R^2$},
\end{aligned}
\end{equation}
where we write $f_u$ instead of $f(u)$ and
$\delta g^i_{uv}:= g^{i}(v)-g^i(u)$, $i=1,2$.
We refer to $\strat$ and $\ito$ seen as functions of three variables 
(better viewed as functions of a two-dimensional simplex) 
as
 \emph{Stratonovich germ} and to the latter
one as \emph{It\^{o} germ} because of their obvious similarity with discrete constructions of the respective integrals in stochastic calculus. The 
terminology of ``germs'', meaning just functions of finite-dimensional simplices, is borrowed from \emph{``germs of rough differential forms''}~\cite{SteTrev18a}, which is in turn inherited from the Rough Paths theory~\cite{gubinelli_controlling_2004}.

In this paper we show that
\begin{itemize}
\item[(A)] if $\Omega$ is an oriented simplex (i.e.\ a triangle), then 
summing either  It\^{o} or Stratonovich germs  
over any sufficiently nice family of its refining triangular partitions
%of the original triangle 
(in particular, dyadic ones) with the appropriately chosen orientation will still lead to the same integral defined by Z\"{u}st, and estimate the rate of convergence (Theorems~\ref{th_exist_discrInt1},~\ref{th_arbpart2}). The respective integral may
be called both It\^{o} and Stratonovich, and in fact generalizes the one-dimensional Young integral.
It is worth emphasizing that this result might seem counterintuitive. In fact the integral should clearly vanish over
degenerate triangles $\Omega$ (i.e.\ those having zero area), while neither the Stratonovich nor the It\^{o} germ possess this property (which we will further call \emph{nonatomicity}), as opposed to the germ $\eta$ defined by~\eqref{eq_germzust1},
nor they are in some obvious way asymptotically close to some nonatomic germ
(unless of course the functions $g_1$ and $g_2$ are differentiable). It is therefore not at all clear how can one expect to be nonatomic a limit of sums of germs which are essentially not so;
\item[(B)] the integral defined in such a way can be extended to a large class of bounded open sets $\Omega\subset\R^2$ having sufficiently small 
box-counting dimension of the topological boundary (Theorem~\ref{th_intOm1}), and in particular can be defined in a very natural way for $\Omega$ a simple polygon (Proposition~\ref{prop_intPoly1});
\item[(C)] when $f$ has a particular form $f(x)=F(x, g(x))$, then the conditions of the existence of the integral extending the classical one (for smooth forms), i.e.\ the requirements
on H\"{o}lder exponents of $g^i$,  may be significantly relaxed at the price of requiring  $F\colon
\R^2\times \R^2\to \R$ to be sufficiently regular (Theorem~\ref{th_irregInt1a})
  by employing  Stratonovich germs. This is however a very particular
feature of Stratonovich but not of It\^{o} summation as can be seen also
in the one-dimensional situation (Remark~\ref{rem_Strat1D}). The resulting Stratonovich type integral is shown to satisfy the classical chain rule (Proposition~\ref{prop_irregIntChrule2}) and may be identified with the  ``second order Riemann-Stieltjes'' integral introduced in~\cite{zust_integration_2011}, 
the respective identification leading to a curious continuity estimate for the degree of H\"{o}lder maps (Remark~\ref{rm_discr_deg1}).
We also give an interpretation of these results in geometric terms of the 
existence
of continuous extensions of De Rham currents associated with the graphs of
smooth maps $g\colon \R^2\to \R^2$ to those associated with 
graphs of H\"{o}lder maps with sufficiently large H\"{o}lder exponents, 
the continuity being intended in the weak (pointwise) topology of currents (Proposition~\ref{prop_irregIntCurr1}). 
\end{itemize}

The key role in the proofs will be played by the observation that both the integral and the Stratonovich germ
are \emph{alternating}, i.e.\ they change sign when the triangle over which they are defined changes the orientation. In fact, our basic instrument will be the natural generalization of the two-dimensional sewing lemma and stability theorem
from~\cite{SteTrev18a} to abstract alternating germs 
(Lemmata~\ref{lm_Vnlimcont1a} and~\ref{lm_Vnlimstab1} respectively).

\section{Notation and preliminaries}

\subsection*{Spaces}
Let $D\subset \R^n$ be an open set. For an $\alpha\in (0,1)$ we will write $C^\alpha(\bar D)$ (abbreviated just to $C^\alpha$ when
there is no possibility of confusion) for the H\"{o}lder space with exponent $\alpha$. For an $f\in C^\alpha(\bar D)$ we denote by
$[\delta f]_\alpha$ its H\"{o}lder seminorm, and $\|f\|_\alpha:=\|f\|_\infty + [\delta f]_\alpha$ its H\"{o}lder norm, where
$\|\cdot\|_\infty$ stands for the usual supremum norm in the space of continuous function $C(\bar D)$ (usually abbreviated to $C$).
The notation $C^1(\bar D)$ (or just $C^1$ for brevity) will stand for the usual space of continuously differentiable functions.

\subsection*{Simplices, chains, germs and rough differential forms}
For an ordered $(k+1)$-uple of points $S=[p_0 p_1 \ldots p_k] \in D^{k+1}$ we write  $\conv S:=\conv\{p_0 p_1 \ldots p_k\}$
and $\diam S$ for the convex envelope and the diameter of the set of points $\{p_0,\ldots, p_k \}$ respectively, and call $S$ an (oriented) \emph{$k$-simplex} in $D$, if $\conv S\subset D$, the set of such simplices being defnoted by $\simp^k(D)$.
For a $k$-simplex $S\in \simp^k(D)$ we denote by $|S|$ its $k$-dimensional volume.
A (real polyhedral) \emph{$k$-chain} in $D$ is an element
of the real vector space $\chain^k$(D) generated by $k$-simplices in $D$.
A $k$-simplex can be identified with the ``geometric'' simplex $\conv S$ with a chosen base point  $p_0$ and the chosen orientation  given by the order of the points in the list, so that $0$-simplices correspond to points, $1$-simplices to oriented segments and  $2$-simplices are pointed oriented triangles.

A \emph{$k$-germ} (of a $k$-differential form in $D$) is a function $\omega\colon \simp^k(D) \to \R$,
\[ S=[p_0 p_1 \ldots p_k] \mapsto \omega_{S} =\omega_{p_0p_1\ldots p_k}.\]
We also often write $\ang{S, \omega}$ instead of $\omega_S$.
A \emph{$k$-cochain} in $D$ is a linear functional $\omega\colon  \chain^k(D) \to \R$,
 \[ C\mapsto  \ang{C, \omega}.\]
For instance, $0$-germs are just functions $p_0\mapsto f(p_0) = f_{p_0} = \ang{[p_0], f}$.

The \emph{boundary} $\partial S$ of an $S\in \simp^k(D)$ is the $(k-1)$-chain defined by
\[  \partial [p_0 p_1 \ldots, p_k]  := \sum_{i=0}^k (-1)^i [p_0 \ldots \hat{p}_i \ldots p_k],\]
the notation $\hat{p}_i$ standing for removal of the respective element from the list.
The operator $\partial$ is naturally extended by linearity to $k$-chains.
The \emph{coboundary} of a $k$-germ $\omega$ is the $(k+1)$-germ $\delta\omega$ defined by duality with the boundary of simplices, namely,
\[
\ang{S,\delta\omega}:= \ang{\partial S,\omega}.
\]
For instance, for a $0$-germ $f$ one has $(\delta f)_{pq}=f_q-f_p$, and for a $1$-germ $\omega$
one has $(\delta \omega)_{pqr}=\omega_{qr}-\omega_{pr}+\omega_{pq}$.

A $k$-germ $\omega$ is called
\begin{itemize}
\item \emph{nonatomic}, if it vanishes on degenerate $k$-simplices $S$ (i.e.\ on those having zero $k$-dimensional volume $|S|=0$). For instance, the germ $\eta$ defined by~\eqref{eq_germzust1} is nonatomic, while the germs $\strat$ and $\ito$ defined by~\eqref{eq_germzust2} are not;
\item \emph{alternating}, if
\[
\ang{[p_0 p_1 \ldots p_k],\delta\omega}:= (-1)^\sigma \ang{[\sigma(p_0)\sigma(p_1) \ldots \sigma(p_k)],\omega}.
\]
for every permutation of vertices $\sigma\colon \{p_0,p_1 \ldots p_k\}\to \{p_0 p_1 \ldots p_k\}$,
$(-1)^\sigma$ standing for the sign of permutation (positive for even and negative for odd permutations).
For instance, among the germs defined by~\eqref{eq_germzust1} and~\eqref{eq_germzust2}, $\strat$ is alternating, while
$\eta$ and $\ito$ are not.
\end{itemize}
Finally, a $k$-germ $\omega$ is called
a \emph{rough differential $k$-form}, if it is continuous (as a function of vertices of a simplex), and both $\omega$ and $\delta\omega$ are nonatomic.
An example of a rough differential $1$-form (written $g^1\d g^2$ for $g^i \in C^{\beta_i}$, $i=1,2$,
with $\beta_1 + \beta_2 >1$) is given by the Young integral over the line segment $[pq]$, that is,
\[
\ang{[pq], g^1\d g^2}:=\int_{[pq]} g^1\d g^2.
\]
An example of a rough differential $2$-form (written $f\d g^1\wedge \d g^2$ for $f\in C^\alpha$, $g^i \in C^{\beta_i}$, $i=1,2$,
with $\alpha+\beta_1 + \beta_2 >2$) is given by the integral defined by R. Z\"{u}st in~\cite{zust_integration_2011}, namely,
\[
\ang{[pqr], f\d g^1\wedge \d g^2}:=\int_{[pqr]} f\d g^1\wedge \d g^2.
\]

The \emph{cup product} (called external product in~\cite{gubinelli_controlling_2004}) between a $k$-germ $\omega$ and a $h$-germ $\tilde{\omega}$ is the $(k+h)$-germ $\omega \cup \tilde{\omega}$ defined by
\[ \ang{[p_0 p_1\ldots p_k p_{k+1} \ldots p_{k+h} ], \omega \cup \tilde{\omega}} :=  \ang{[p_0 p_1\ldots p_k], \omega} \ang{ [p_k p_{k+1} \ldots p_{k+h} ], \tilde{\omega}}.\]
The cup product is associative but in general not commutative, and the following  Leibniz rule holds~\cite{SteTrev18a}: for $\omega \in \germ^k(\O)$, $\tilde\omega \in \germ^{h}(\O)$ one has
\begin{equation}
\label{eq:identity-leibniz-cup1}
 \delta ( \omega \cup \tilde\omega) = (\delta \omega) \cup \tilde{\omega} + (-1)^{k} \omega \cup (\delta \tilde{\omega}). \end{equation}

\section{Estimates on germs}

We start with the following useful algebraic lemma.

\begin{lemma}\label{lm_detcup1}
One has
\begin{equation}\label{eq_detcup1a}
\begin{aligned}
\frac 1 2\det \bra{ \begin{array}{ll} \delta g^1_{p q} & \delta g^1_{pr} \\
\delta g^2_{pq} & \delta g^2_{pr} \end{array}}
&= \frac 1 2\det \bra{ \begin{array}{ll} \delta g^1_{p q} & \delta g^1_{qr} \\
\delta g^2_{pq} & \delta g^2_{qr} \end{array}}
= \frac 1 2\det \bra{ \begin{array}{ll} \delta g^1_{r q} & \delta g^1_{pr} \\
\delta g^2_{rq} & \delta g^2_{pr} \end{array}}
\\
&= \A (\delta g^1\cup \delta g^2)_{pqr},
\end{aligned}
\end{equation}
where $\A$ stands for the antisymmetrization operator
\[
\A (\phi\cup\psi) := \frac 1 2 \left( \phi\cup\psi- \psi\cup\phi \right).
\]
In particular,
\begin{equation}\label{eq_detcup1b}
\ito_{pqr} = (f \cup \A (\delta g^1\cup \delta g^2))_{pqr}.
\end{equation}
\end{lemma}

\begin{proof}
It suffices to calculate
\begin{align*}\det \bra{ \begin{array}{ll} \delta g^1_{p q} & \delta g^1_{pr} \\
\delta g^2_{pq} & \delta g^2_{pr} \end{array}} -
\det \bra{ \begin{array}{ll} \delta g^1_{p q} & \delta g^1_{qr}
\\
\delta g^2_{pq} & \delta g^2_{qr} \end{array}}
& =
\det \bra{ \begin{array}{ll} \delta g^1_{p q} & \delta g^1_{pr}- \delta g^1_{qr} \\
\delta g^2_{pq} & \delta g^2_{pr} - \delta g^2_{qr}\end{array}}\\
& =
\det \bra{ \begin{array}{ll}  \delta g^1_{p q} & \delta g^1_{pq} \\
\delta g^2_{pq} & \delta g^2_{pq} \end{array}} =0
\end{align*}
to show the first equality in~\eqref{eq_detcup1a}; the third one follows then from the definition of $\A$. The second equality is quite analogous
from
\begin{align*}\det \bra{ \begin{array}{ll} \delta g^1_{p q} & \delta g^1_{pr} \\
\delta g^2_{pq} & \delta g^2_{pr} \end{array}} -
\det \bra{ \begin{array}{ll} \delta g^1_{r q} & \delta g^1_{pr} \\
\delta g^2_{rq} & \delta g^2_{pr} \end{array}}
& =
\det \bra{ \begin{array}{ll} \delta g^1_{p q} -  \delta g^1_{r q}& \delta g^1_{pr} \\
\delta g^2_{pq}  - g^2_{rq}& \delta g^2_{pr} \end{array}}\\
& =
\det \bra{ \begin{array}{ll}  \delta g^1_{qp} & \delta g^1_{pq} \\
\delta g^2_{qp} & \delta g^2_{pq} \end{array}} =0,
\end{align*}
concluding the proof.
\end{proof}

Notice that $\A (\delta g^1\cup \delta g^2)= \delta \eta$ with $\eta= \frac 1 2 \bra{g^1 \delta g^2 - g^2 \delta g^1}$.

\begin{lemma}\label{lm_estomgam1}
One has
\begin{eqnarray}
\label{eq_estgam1a}
|\ito_{pqr}-\strat_{pqr}| \le 2[\delta f]_{\alpha}[\delta g^1]_{\beta_1}[\delta g^2]_{\beta_2}
 \diam([pqr])^{\alpha+\beta_1+\beta_2}\\
\label{eq_estom1a}
|\strat_{pqr}| \le %\frac 1 2
\norm{f}_\infty [\delta g^1]_{\beta_1} [\delta g^2]_{\beta_2} \diam([pqr])^{\beta_1+\beta_2}\\
\label{eq_estdom1b}
|\delta \strat_{pqrs} |  \le 8 [\delta f]_{\alpha}[\delta g^1]_{\beta_1} [\delta g^2]_{\beta_2} \diam([pqrs])^{\alpha+\beta_1+\beta_2 }
 \end{eqnarray}
 and $\strat$ is alternating, namely,
\[ \strat_{pqr} = \strat_{rpq} = \strat_{qrp} = -\strat_{rqp} = -\strat_{prq} = - \strat_{prq}.\]
\end{lemma}

\begin{remark}\label{rem_estom1a} Clearly,~\eqref{eq_estom1a}
holds even for every $f\in M$, where $M$ stands for the space of bounded
(not necessarily measurable) functions over $\R^2$ equipped with the supremum norm (still denoted by $\|\cdot\|_\infty$).
\end{remark}

\begin{proof}
The estimate~\eqref{eq_estom1a} as well as the alternating property of $\strat$ is immediate from the definition
of $\strat$.
To show~\eqref{eq_estgam1a}, we calculate
\begin{align*}
|\ito_{pqr}- \strat_{pqr}| & =
\frac 1 2 \left|\bra{\frac{f_{p} +f_{q} +f_r}{3} -f_p} \det \bra{ \begin{array}{ll} \delta g^1_{p q} & \delta g^1_{pr}\\
\delta g^2_{pq} & \delta g^2_{pr} \end{array}}\right| \\
& \leq [\delta f]_{\alpha}[\delta g^1]_{\beta_1}[\delta g^2]_{\beta_2}\diam (pqr)^{\alpha+\beta_1+\beta_2}
\end{align*}
as claimed.
Thus,~\eqref{eq_estdom1b} would follow once one proves
\begin{eqnarray}
\label{eq_estdgam1a}
|\delta \ito_{pqrs} |
\le %C
[\delta f]_{\alpha}[\delta g^1]_{\beta_1} [\delta g^2]_{\beta_2} \diam([pqrs])^{\alpha+\beta_1+\beta_2 }.
 \end{eqnarray}
To show the latter inequality,
we use Lemma~\ref{lm_detcup1}: namely, by~\eqref{eq_detcup1b} one has
\begin{equation}
\label{eq_detcup1c}
\ito = \frac 1 2 \left( (f \cup \delta g^1\cup \delta g^2)- (f \cup \delta g^2\cup \delta g^1)\right).
\end{equation}
Therefore, using the fact that
\[
\delta (\delta g^1\cup \delta g^2) = \delta g^1\cup \delta( \delta g^2) - \delta(\delta g^1)\cup \delta g^2=0,
\]
and analogously $\delta (\delta g^2\cup \delta g^1)=0$, from~\eqref{eq_detcup1c} we get
\begin{equation}\label{eq_estdgam1b}
\begin{aligned}
\delta\ito & = \frac 1 2 \left( \delta (f \cup \delta g^1\cup \delta g^2)- \delta (f \cup \delta g^2\cup \delta g^1)\right) \\
& =
\frac 1 2 \left( (\delta f \cup \delta g^1\cup \delta g^2)- (\delta f \cup \delta g^2\cup \delta g^1)\right).
\end{aligned}
\end{equation}
Since clearly,
\begin{align*}
|(\delta f \cup \delta g^1\cup \delta g^2)_{pqrs}| & = |(\delta f)_{pq}|\cdot |(\delta g^1)_{qr}| \cdot |(\delta g^2)_{rs}|\\
& \leq  [\delta f]_{\alpha}[\delta g^1]_{\beta_1} [\delta g^2]_{\beta_2} \diam([pqrs])^{\alpha+\beta_1+\beta_2 },
\end{align*}
and analogously
\[
|(\delta f \cup \delta g^2\cup \delta g^1)_{pqrs}|
\leq  [\delta f]_{\alpha}[\delta g^1]_{\beta_1} [\delta g^2]_{\beta_2} \diam([pqrs])^{\alpha+\beta_1+\beta_2 },
\]
from~\eqref{eq_estdgam1b} we get~\eqref{eq_estdgam1a},
and therefore~\eqref{eq_estdom1b},
%with $C:=1$,
hence concluding the proof.
\end{proof}

Later in section~\ref{sec_strat_irreg1}
we will need also the following curious algebraic identity which is a peculiar property of only the Stratonovich germ
$\strat$ and not of the It\^{o} germ $\ito$, and
could have been also used
for an alternative proof of~\eqref{eq_estdom1b} in Lemma~\ref{lm_estomgam1}.

\begin{lemma}\label{lm_StratDet1}
One has
\[
(\delta\strat)_{pqrs} =
\frac 1 6\det \bra{ \begin{array}{lll}
\delta f_{p q} & \delta f_{pr} & \delta f_{ps} \\
\delta g^1_{p q} & \delta g^1_{pr} &\delta g^1_{ps} \\
\delta g^2_{pq} & \delta g^2_{pr} &\delta g^2_{ps}
\end{array}}.
\]
\end{lemma}

\begin{proof}
By Lemma~\ref{lm_detcup1} one has
\begin{align*}
6& \strat_{pqr} \\
& \quad = f_p \det \bra{ \begin{array}{ll} \delta g^1_{p q} & \delta g^1_{qr} \\
\delta g^2_{pq} & \delta g^2_{qr} \end{array}} + f_q \det \bra{ \begin{array}{ll} \delta g^1_{p q} & \delta g^1_{qr} \\
\delta g^2_{pq} & \delta g^2_{qr} \end{array}}
+ f_r \bra{ \begin{array}{ll} \delta g^1_{p q} & \delta g^1_{qr} \\
\delta g^2_{pq} & \delta g^2_{qr} \end{array}}\\
& \quad = (f \cup \delta g^1 \cup \delta g^2 -
f \cup \delta g^2 \cup \delta g^1)_{pqr}
%\\
%& \quad
+
(\delta g^1 \cup f \cup \delta g^2 -  \delta g^2 \cup f \cup \delta g^1)_{pqr} \\
& \qquad + (\delta g^1 \cup \delta g^2 \cup f - \delta g^2 \cup \delta g^1 \cup f)_{pqr}.
\end{align*}
Hence,
\[\begin{split} 6 (\delta \strat)_{pqrs} &=   (\delta f \cup \delta g^1 \cup \delta g^2 - \delta f \cup \delta g^2 \cup \delta g^1)_{pqrs} + \\
&\qquad (- \delta g^1 \cup \delta f \cup \delta g^2 + \delta g^2 \cup \delta f \cup \delta g^1)_{pqrs} + \\
& \qquad (\delta g^1 \cup \delta g^2 \cup \delta  f - \delta g^2 \cup \delta g^1 \cup \delta f)_{pqrs}\\
& = \det\bra{ \begin{array}{lll}
\delta f_{pq} & \delta f_{qr} & \delta f_{rs}\\
\delta g^1_{pq} & \delta g^1_{qr} & \delta g^1_{rs}\\
\delta g^2_{pq} & \delta g^2_{qr} & \delta g^2_{rs}\\
\end{array}} = \frac 1 6 \det\bra{ \begin{array}{lll}
\delta f_{pq} & \delta f_{pr} & \delta f_{ps}\\
\delta g^1_{pq} & \delta g^1_{pr} & \delta g^1_{ps}\\
\delta g^2_{pq} & \delta g^2_{pr} & \delta g^2_{ps}\\
\end{array}},
\end{split}
\]
where the latter identity follows by adding the first column to the second one and subsequently the second column to the third one.
\end{proof}

\section{Riemann summation over dyadic partitions}\label{sec_dyaInt1}

Recall~\cite{SteTrev18a} the \emph{dyadic decomposition} of a $2$-simplex $[p_0 p_1 p_2]\in \simp^2(D)$
\[ \dya [p_0 p_1 p_2] := [q_0 q_1 q_2] + [q_1 q_0 p_2] + [q_2 p_1 q_0]+ [p_0 q_2 q_1],\]
where $q_i := (p_{j} + p_\ell)/2$ for $\cur{i,j,\ell} = \cur{0,1,2}$. Write also $\cut [p_0 p_1] := [p_0 q] +[q p_1]$ and
  $\fill [p_0 p_1] := [p_0 q p_1]$, with $q := (p_0 +p_1)/2$ (naturally extended to chains).

For $n \in \N$ define the $n$-th \emph{Stratonovich sum} $\strat^n$, the \emph{side corrector} $S^n$ as well as the
\emph{It\^{o} sum} $\ito^n$
respectively by the formulae
\begin{equation}\label{def_omnSn}
\begin{aligned}
 \strat^n_{pqr}  & := \ang{ \dya^n [pqr], \strat}, \quad
%S^n_{pq} := \sum_{i=0}^{n} \ang{ \fill \cut^i [pq], \strat}.
S^n_{pq} := \sum_{i=0}^{n-1} \ang{ \fill \cut^i [pq], \strat},\\
 \ito^n_{pqr}  & := \ang{ \dya^n [pqr], \ito}.
\end{aligned}
\end{equation}

\begin{lemma}\label{lm_Vnlimcont1}
One has
\begin{eqnarray}
\label{eq_incrSn1}
|S^{n+1}_{pq} - S^n_{pq}| \le C\norm{f}_\infty [\delta g^1]_{\beta_1} [\delta g^2]_{\beta_2} \diam([pq])^{\beta_1+\beta_2} 2^{n(1- \beta_1-\beta_2)},\\
\label{eq_incrVn1}
 \begin{aligned} |\langle [pqr], (\strat^n - \delta S^n) & - (\strat^{n+1} - \delta S^{n+1})\rangle| \\
  & \le C [\delta f]_{\alpha}[\delta g^1]_{\beta_1} [\delta g^2]_{\beta_2}  \diam([pqr])^{\alpha+\beta_1+\beta_2 } 2^{n(2-\alpha- \beta_1-\beta_2)}
  \end{aligned}
\end{eqnarray}
with $C>0$ a universal constant. %depending only on $\alpha+\beta_1+\beta_2$.
In particular, if $\alpha+\beta_1+\beta_2>2$, then
\begin{equation}\label{eq_defSV}
\begin{aligned}
 S_{pq}& :=  \lim_{n \to \infty} S^n_{pq},\\
V_{pqr} & := \lim_{n \to \infty} \strat^n_{pqr} = \lim_{n \to \infty} (\strat^n _{pqr}- \delta S^n_{pqr}) + \delta S^n_{pqr}
\end{aligned}
\end{equation}
are well defined continuous alternating germs
with
\[ S_{pq} \colon C^0 \times C^{\beta_1} \times C^{\beta_2}  \to \R, \quad  \quad  V_{pqr} \colon C^{\alpha} \times C^{\beta_1} \times C^{\beta_2} \to \R\]
continuous and
\begin{eqnarray}
\label{eq_estlimS1a}
| S^n_{pq}-S_{pq}|\leq  \norm{f}_\infty [\delta g^1]_{\beta_1} [\delta g^2]_{\beta_2} \diam([pq])^{\beta_1+\beta_2} 2^{n(1- \beta_1-\beta_2)},\\
\label{eq_estlimom1a}
\begin{aligned}
|\strat^n_{pqr}-V_{pqr} & -\delta(S^n-S)_{pqr}| \leq \\
&  C [\delta f]_{\alpha}[\delta g^1]_{\beta_1} [\delta g^2]_{\beta_2}  \diam([pqr])^{\alpha+\beta_1+\beta_2 } 2^{n(2-\alpha- \beta_1-\beta_2)},
\end{aligned}\\
\label{eq_estlimom1b}
|\strat^n_{pqr}-V_{pqr}| \leq C \|f\|_{\alpha}[\delta g^1]_{\beta_1} [\delta g^2]_{\beta_2}  \diam([pq])^{\beta_1+\beta_2} 2^{n(1- \beta_1-\beta_2)}.
\end{eqnarray}
\end{lemma}

\begin{remark}\label{rem_estom1b}
As one easily deduces from the proof, in view of the Remark~\ref{rem_estom1a}, one has, with the notation
of the latter, that
in fact $S_{pq}$ itself may be defined over the larger space
$M\times C^{\beta_1}\times C^{\beta_2}$ and is
continuous there when just $\beta_1+\beta_2>1$.
\end{remark}

\begin{proof}
We apply Lemma~\ref{lm_Vnlimcont1a} to our germ $\strat$ (which is continuous
and alternating by construction) recalling
that it satisfies both ~\eqref{eq_omgam1} and~\eqref{eq_omgam2} with
\begin{align*}
\gamma_1 & :=\beta_1+\beta_2>1, \quad C_1  :=\norm{f}_\infty [\delta g^1]_{\beta_1} [\delta g^2]_{\beta_2},\\
\gamma_2 & :=\alpha+\beta_1+\beta_2>2, \quad C_2  :=8[f]_\alpha [\delta g^1]_{\beta_1} [\delta g^2]_{\beta_2}
\end{align*}
in view of Lemma~\ref{lm_estomgam1}.
This gives~\eqref{eq_incrSn1} and~\eqref{eq_incrVn1}, as well as
the existence of limit germs
alternating continuous $S$ and $V$ in~\eqref{eq_defSV}
satisfying~\eqref{eq_estlimS1a},~\eqref{eq_estlimom1a} and~\eqref{eq_estlimom1b}.
Finally, the continuity of  $S_{pq}$ (with fixed $[pq]$) as a functional
follows from~\eqref{eq_incrSn1} and implies the continuity of
$\delta S_{pqr} \colon C^0 \times C^{\beta_1} \times C^{\beta_2}  \to \R$. Continuity of
\[
V_{pqr}-\delta S_{pqr} := \lim_{n \to \infty} (\strat^n _{pqr}- \delta S^n_{pqr}) \colon C^{\alpha} \times C^{\beta_1} \times C^{\beta_2}\to \R
\]
follows from~\eqref{eq_incrVn1}, hence implying the continuity of $V$, and therefore concluding the proof.
\end{proof}

We will need also the following Lemma already formulated
in~\cite[example~4.7]{SteTrev18a}.
% ``towards rough differential forms''.

\begin{lemma}
If $\beta_1 = \beta_2 = 1$, then
\[
V_{pqr} =\int_{[pqr]} f \d g^1 \wedge \d g^2 = \int_{[pqr]} f \det( \nabla g^1, \nabla g^2) .
\]
\end{lemma}

% \begin{proof}
% Already noticed in~\cite{SteTrev18a}. % ``towards rough differential forms''.
% \end{proof}

We are now at a position to prove the first principal result of this paper.

\begin{theorem}\label{th_exist_discrInt1}
If $\alpha+\beta_1+\beta_2>2$, then
\begin{eqnarray}
\label{eq_limom1}
\lim_{n \to \infty} \strat^n_{pqr} &=&
\displaystyle
 \int_{[pqr]} f \d g^1 \wedge \d g^2\\
\label{eq_limgam1}
&=&\displaystyle  \lim_{n \to \infty} \ito^n_{pqr}. %= \int_{[pqr]} f \d g^1 \wedge \d g^2.
\end{eqnarray}
In particular, the latter integral is
\begin{itemize}
\item[(A)] nonatomic, i.e.\
\[
\int_{[pqr]} f \d g^1 \wedge \d g^2=0 \quad \mbox{when $|[pqr]|=0$},
\]
\item[(B)] continuous and alternating in $[pqr]$, and
\item[(C)] additive, in the sense that when
\[
[pqr]=\sum_{i=1}^k \Delta_i + N +\partial R,
\]
where $\Delta_i$ are oriented $2$-simplices, $N$ is a polyhedral $2$-chain consisting of
degenerate $2$-simplices (i.e.\ having area zero), and $R$ is a polyhedral $3$-chain in $\R^2$, then
\[
\int_{[pqr]} f \d g^1 \wedge \d g^2=
\sum_{i=1}^k\int_{\Delta_i}
f \d g^1 \wedge \d g^2.
\]
\end{itemize}
Moreover,
 \begin{equation}
\label{eq_estlimom1}
\left| \strat^n_{pqr} - \int_{[pqr]} f \d g^1 \wedge \d g^2\right| \le C \norm{f}_{\alpha} [\delta g^1]_{\beta_1} [\delta g^2]_{\beta_2} \diam([pqr])^{\beta_1+\beta_2} 2^{n(1- \beta_1-\beta_2)},
\end{equation}
\begin{equation}
\label{eq_estlimgam1}
\abs{ \ito^n_{pqr} - \int_{[pqr]} f \d g^1 \wedge \d g^2} \le C \norm{f}_{\alpha} [\delta g^1]_{\beta_1} [\delta g^2]_{\beta_2} \diam([pqr])^{\beta_1+\beta_2} 2^{n(1- \beta_1-\beta_2)}
\end{equation}
for some $C=C(\alpha, \beta_1, \beta_2)>0$.
\end{theorem}

\begin{proof}
By Lemma~\ref{lm_Vnlimcont1}, the limit
\[
V_{pqr}:=\lim_{n \to \infty} \strat^n_{pqr}
\]
exists and is a continuous multilinear functional over $C^{\alpha} \times C^{\beta_1} \times C^{\beta_2}$, and
\[
V_{pqr}(f, g^1, g^2) =\int_{[pqr]} f \d g^1 \wedge \d g^2 := \int_{[pqr]} f \det( \nabla g^1, \nabla g^2)\, dx
\]
when $f\in C^0$, $g^i\in C^1$, $i=1,2$. However the unique continuous extension of the latter functional defined over
$C^0 \times C^1 \times C^1$ to $C^{\alpha} \times C^{\beta_1} \times C^{\beta_2}$ is the Z\"{u}st integral, which implies the claim~\eqref{eq_limom1},~\eqref{eq_estlimom1}.
Properties~(A),~(B) and~(C) are now in fact the properties of the Z\"{u}st integral
(theorem~4.10 of~\cite{SteTrev18a} where they are stated by saying that
the Z\"{u}st germ~\eqref{eq_germzust1} is \emph{sewable}).

The claims~\eqref{eq_limgam1},~\eqref{eq_estlimgam1} follow now from~\eqref{eq_estgam1a}.
\end{proof}

\begin{remark}\label{rem_betterconv1}
One also has the inequality~\eqref{eq_estlimom1a} which can be rewritten, in view of the above Theorem~\ref{th_exist_discrInt1}
as
\begin{equation}\label{eq_betterconv1}
\begin{aligned} \bigg|\strat^n_{pqr}  - \int_{[pqr]}
   %\right.
   &
   %\left.
                     f \d g^1 \wedge \d g^2 - \delta (S^n -S)_{pqr}
\bigg|\\
&
\le C \norm{f}_{\alpha} [\delta g^1]_{\beta_1} [\delta g^2]_{\beta_2}
\diam([pqr])^{\alpha+\beta_1+\beta_2} 2^{n(2-\alpha- \beta_1-\beta_2)}
\end{aligned}
\end{equation}
with $C=C(\alpha, \beta_1, \beta_2)>0$.
Thus, in order to improve the convergence rate one should better approximate $S^n -S$.
This is the case e.g.\ when on the boundary of $[pqr]$ either $f$ is null or one of the $g^i$ is constant: in fact, in these cases $S^n=0$ and hence also $S=0$.
%
%What happens when $f=1$ is constant? When $g^1$ and $g^2$ are regular?
\end{remark}

\begin{remark}
The $2$-germ $f\cup \delta g^1\cup\delta g^2$ in general does not provide an integral even when $f$, $g^1$ and $g^2$ are smooth.
In fact, let $f=1$, $g^i(x_1, x_2):= x_i$, $i=1, 2$, $p=(0,0)$, $q=(1,0)$, $r=(0,1)$. Then
$\langle \dya^n [pqr],  f\cup \delta g^1\cup\delta g^2\rangle \to 2|[pqr]|$ while
$\langle \dya^n [pqr], f\cup \delta g^2\cup\delta g^1\rangle \to 0$ as $n\to \infty$, i.e.\ the limit is not alternating.
\end{remark}

\begin{remark}
As mentioned in the introduction, the above theorem allows to define  the integral of a differential $2$-form $\omega = f \d g^1 \wedge \d g^2$ on $\R^n$ over a parameterized H\"older surface $\varphi\colon\Omega\to \R^n$,  $\varphi(x) = (\varphi^i(x))_{i=1}^n$, letting
\[ \int_{\varphi([pqr])} f \d g^1 \wedge \d g^2 := \int_{[pqr]} (f
\circ \varphi) \d (g^1\circ \varphi) \wedge \d (g^2 \circ \varphi),\]
provided that $f\in C^{\alpha}(\R^n)$, $g^i \in C^{\beta_i}(\R^n)$, $i=1, 2$, $\varphi \in C^{\gamma}(\R^2; \R^n)$ with
\[ \gamma(\alpha+\beta_1+\beta_2)>2.\]
Notice however that the above integral differs from the integral obtained partitioning the triangle $[\varphi(p) \varphi(q) \varphi(r)]$  with an order $\diam([pqr])^{\gamma(\beta_1+\beta_2)}$ and not $\diam([pqr])^{\gamma(\alpha+\beta_1+\beta_2)}$, see \cite[proposition~4.29]{SteTrev18a}.
\end{remark}

\begin{corollary}\label{co_dg1dg2is0a}
If there is an $h\in C^{\beta_3}$, $\beta_3\in (0,1]$, such that both
$g^1$ and $g^2$ are $h$-differentiable in the sense
\[
(\delta g^i)_{pq}= a_p^i (\delta h)_{pq}  + o(|p-q|), \quad i=1,2
\]
for every $p\in D$ as $q\to p$, and, moreover,
\begin{eqnarray}
\label{eq_errest1a}
|(\delta g^i)_{pq}- a_p^i (\delta h)_{pq}|\leq C|p-q|^{1+\gamma_i}
\end{eqnarray}
for some $\gamma_i> 1-\beta_3$, $i=1,2$, and $C>0$, then
$\d g^1\wedge \d g^2=0$ in the sense
 \begin{equation*}
 \int_{[pqr]} f \d g^1 \wedge \d g^2=0
\end{equation*}
for every $f\in C^\alpha$ with $\alpha+\beta_1+\beta_2>2$ and every $[pqr]\subset D$.
\end{corollary}

\begin{proof}
Let
\[
\rho_{pq}^i:= (\delta g^i)_{pq}- a_p^i (\delta h)_{pq}.
\]
Then
\begin{equation}\label{eq_detest1a}
\begin{aligned}
\frac 1 2\det \bra{ \begin{array}{ll} \delta g^1_{p q} & \delta g^1_{pr} \\
\delta g^2_{pq} & \delta g^2_{pr} \end{array}}
&= \frac 1 2 a_p^1 a_p^2 \det \bra{ \begin{array}{ll} \delta h_{p q} & \delta h_{pr} \\
\delta h_{pq} & \delta h_{pr} \end{array}} +
\frac 1 2 a_p^1 a_p^2 \det \bra{ \begin{array}{ll} \rho^1_{p q} & \delta h_{pr} \\
\rho^2_{pq} & \delta h_{pr} \end{array}} \\
& \quad + \frac 1 2 a_p^1 a_p^2 \det \bra{ \begin{array}{ll} \delta h_{p q} & \rho^1_{pr} \\
\delta h_{pq} & \rho^2_{pr} \end{array}}
+
\frac 1 2\det \bra{ \begin{array}{ll} \rho^1_{p q} & \rho^1_{pr} \\
\rho^2_{pq} & \rho^2_{pr} \end{array}}.
\end{aligned}
\end{equation}
Letting $\gamma:=\gamma_1\wedge \gamma_2$, from~\eqref{eq_errest1a} we get
\begin{align*}
\left|\det \bra{ \begin{array}{ll} \rho^1_{p q} & \delta h_{pr} \\
\rho^2_{pq} & \delta h_{pr}\end{array}}\right| & \leq 2C [h]_{\beta_3}\diam([pqr])^{1+\gamma+\beta_3},\\
\left|\det \bra{ \begin{array}{ll} \delta h_{p q} & \rho^1_{pr} \\
\delta h_{pq} & \rho^2_{pr} \end{array}}\right| & \leq 2C [h]_{\beta_3}\diam([pqr])^{1+\gamma+\beta_3},\\
\left|\det \bra{ \begin{array}{ll} \rho^1_{p q} & \rho^1_{pr} \\
\rho^2_{pq} & \rho^2_{pr} \end{array}}\right| & \leq 2C \diam([pqr])^{2+\gamma_1+\gamma_2},
\end{align*}
so that by~\eqref{eq_detest1a} one has
\[
|\strat_{pqr}|\leq 2C\|f\|_{\infty}(\|a^1\|_\infty\|a^2\|_\infty[h]_{\beta_3} \diam([pqr])^{1+\gamma +\beta_3} +
\diam([pqr])^{2+\gamma_1+\gamma_2}),
\]
which concludes the proof since $1+\gamma+\beta_3>2$ and $2+\gamma_1+\gamma_2>2$.
\end{proof}

\begin{remark}\label{rm_dg1dg2is0b}
In particular, if $g^1$ is $g^2$-differentiable and, moreover,
\begin{eqnarray}
\label{eq_errest1b}
|(\delta g^1)_{pq}- a_p (\delta g^2)_{pq}|\leq C|p-q|^{1+\gamma}
\end{eqnarray}
for some $\gamma> 1-\beta_2$ and $C>0$, then
$\d g^1\wedge \d g^2=0$.
\end{remark}

\section{General partitions}

Theorem~\ref{th_exist_discrInt1} shows that the integral $\int_{[pqr]}   f \d g^1 \wedge \d g^2$ can be obtained as a limit
of sums of 
the Stratonovich germs over dyadic partitions of the the simplex $[pqr]$. Here we show that it can be obtained by a similar summation of such germs over more general partitions.

\begin{theorem}\label{th_arbpart2}
Assume that the simplex $[pqr]$ be partitioned in a finite number of disjoint simplices $\{\Delta_i\}_{i=1}^N$
not belonging to the sides of $[pqr]$
so that
\begin{equation}\label{eq_genpart1}
[pqr]-\sum_{i=1}^N \Delta_i =\partial P + \sum_{j=1}^M Q_j,
\end{equation}
where $P\in \chain^3(D)$ and each $Q_j\in \simp^2(D)$ is a degenerate simplex reduced to a line segment belonging
to some side of $[pqr]$ such that two sides of each $Q_j$ are sides of some $\Delta_i$ (with opposite direction).
Then
\begin{equation}\label{eq_arbpart2}
\begin{aligned}
\left|\sum_{i=1}^N \right. & \left. \ang{\Delta_i,\strat} - \int_{[pqr]}   f \d g^1 \wedge \d g^2\right| \\
& \le C \norm{f}_{\alpha} [\delta g^1]_{\beta_1} [\delta g^2]_{\beta_2}
\left(\sum_{i=1}^N \diam(\Delta_i)^{\alpha+\beta_1+\beta_2} +\sum_{j=1}^M \diam(Q_j)^{\beta_1+\beta_2}\right).
\end{aligned}
\end{equation}
\end{theorem}

\begin{proof}
The estimate~\eqref{eq_betterconv1} applied to each $\Delta_i$ %and $Q_j$
with $n:=0$ gives
\begin{align*}
\left| \ang{\Delta_i,\strat} - \int_{\Delta_i}   f \d g^1 \wedge \d g^2\right. & \left.
- \ang{\Delta_i,\delta(S^0-S)}\right| \\
& \le C \norm{f}_{\alpha} [\delta g^1]_{\beta_1} [\delta g^2]_{\beta_2}
\diam(\Delta_i)^{\alpha+\beta_1+\beta_2}.
\end{align*}
Summing the latter estimates over $i=1,\ldots, N$,
and recalling that
\[
\int_{[pqr]}   f \d g^1 \wedge \d g^2=\sum_{i=1}^N \int_{\Delta_i}   f \d g^1 \wedge \d g^2
\]
in view of~\eqref{eq_genpart1},
we get
\begin{equation}\label{eq_arbpart1}
\begin{aligned}
\left|\sum_{i=1}^N\ang{\Delta_i,\strat} - \int_{[pqr]}   f \d g^1 \wedge \d g^2\right. & \left.  - \sum_{j=1}^M \ang{q_j,S^0-S}\right| \\
 & \le C \norm{f}_{\alpha} [\delta g^1]_{\beta_1} [\delta g^2]_{\beta_2}
\sum_{i=1}^N \diam(\Delta_i)^{\alpha+\beta_1+\beta_2},
\end{aligned}
\end{equation}
where $q_j\in \simp^1(D)$ is the side of $Q_j$ which is not a side of any $\Delta_i$:
%,since $\strat$ is alternating:
in fact, %as in the proof of Lemma~\ref{lm_Vnlimcont1},
when summing the terms
\[
\ang{\Delta_i,\delta(S^0-S)}= \ang{\partial \Delta_i,S^0-S}
\]
over $i$, we have that every side of some simplex of the partition
which is not one of $q_j$ (i.e.\ does not belong to a side of $[pqr]$)
appears in this sum twice and in opposite directions, and hence is cancelled out from this sum.
Moreover, from~\eqref{eq_estlimS1a} applied with $q_j$ instead of $[pq]$ and $n:=0$ we get
\begin{align*}
| \ang{q_j, S^0-S}| & \leq  \norm{f}_\infty [\delta g^1]_{\beta_1} [\delta g^2]_{\beta_2} \diam(q_j)^{\beta_1+\beta_2},
%\\
%& =  \norm{f}_\infty [\delta g^1]_{\beta_1} [\delta g^2]_{\beta_2} \diam(Q_j)^{\beta_1+\beta_2},
\end{align*}
which together with~\eqref{eq_arbpart1} gives~\eqref{eq_arbpart2} since $\diam q_j = \diam Q_j$.
%\begin{equation}\label{eq_arbpart2}
%\begin{aligned}
%\left|\sum_{i=1}^N\ang{\Delta_i,\strat} \right. & \left.- \int_{[pqr]}   f \d g^1 \wedge \d g^2\right| \\
%& \le C \norm{f}_{\alpha} [\delta g^1]_{\beta_1} [\delta g^2]_{\beta_2}
%\left(\sum_{i=1}^N \diam(\Delta_i)^{\alpha+\beta_1+\beta_2} +\sum_{j=1}^M \diam(Q_j)^{\beta_1+\beta_2}\right).
%\end{aligned}
%\end{equation}
\end{proof}

%\appendix

\section{Integration over general domains}\label{sec_intOm1}

In section~\ref{sec_dyaInt1} we defined the integral of the ``rough differential form'' $f\d g^1 \wedge \d g^2$ over
an arbitrary oriented simplex $[pqr]$ in the domain of definition of $f$ and $g$. Here we show how the latter
can be naturally extended to more general domains $\Omega\subset\R^2$.

First, consider the case when $\Omega$ is an oriented simple (i.e.\ not self-intersecting) polygon with vertices $a_0,\ldots, a_k$, enumerated according to the orientation of $\Omega$ (say, counterclockwise). We will write in this case $\Omega=[a_0\ldots a_k]$. 
Consider the triangulation of $\Omega$ in two-dimensional simplices
$\{\Delta_i\}_{i=1}^m$ oriented in the same direction of $\Omega$. 
We set then by definition
\begin{equation}\label{eq_delPoly1}
\int_{[a_0\ldots a_k]} f\d g^1 \wedge \d g^2 := \sum_{i=1}^m \int_{\Delta_i} f\d g^1 \wedge \d g^2. 
\end{equation}
The following statement is valid.

\begin{proposition}\label{prop_intPoly1}
	Under conditions of Theorem~\ref{th_exist_discrInt1} for every $b\in \R^2$ one has
	\begin{equation}\label{eq_delPoly2}
	\int_{[a_0\ldots a_k]} f\d g^1 \wedge \d g^2 = \sum_{j=0}^k \int_{[a_ja_{j+1}b]} f\d g^1 \wedge \d g^2, 
	\end{equation}
	where $k+1:=0$. In particular, the definition~\eqref{eq_delPoly1} is correct (i.e.\ independent on the particular triangulation
	$\{\Delta_i\}$), 
	the above integral  %$\int_{[a_0\ldots a_k]} f\d g^1 \wedge \d g^2$ is %continuous in the vertices $a_0,\ldots,a_k$, 
	is alternating
	(i.e. preserves/resp.\ changes sign with odd/resp.\ even permutation of the vertices), 
	nonatomic (i.e. zero on polygons of zero area),
	and the map
	\[
	(f,g^1, g^2) \mapsto \int_{[a_0\ldots a_k]} f\d g^1 \wedge \d g^2
	\]	
	is a continuous multilinear functional over $C^\alpha\times C^{\beta_1}\times C^{\beta_2}$ continuous also in the vertices $a_0,\ldots,a_k$ (i.e.\ continuous with respect to the simultaneous convergence of both functions involved and 
	of the vertices).
\end{proposition}

\begin{proof}
	Writing $\Delta_i:=[\alpha_i^1\alpha_i^2\alpha_i^3]$, one has
	\begin{align*}
%	   \partial\sum_{i=1}^m [b\alpha_i^1\alpha_i^2\alpha_i^3] &= 
	       \sum_{i=1}^m\partial [b\alpha_i^1\alpha_i^2\alpha_i^3] 
	    %\\
	         %&
	         = \sum_{i=1}^m [\alpha_i^1\alpha_i^2\alpha_i^3] - \sum_{i=1}^m [b\alpha_i^2\alpha_i^3] + \sum_{i=1}^m [b\alpha_i^1\alpha_i^3] - \sum_{i=1}^m [b \alpha_i^1\alpha_i^2],
	\end{align*}
so that taking into account~\eqref{eq_delPoly1}, and recalling that
\[
\left\langle \partial [pqrs], f\d g^1 \wedge \d g^2 \right\rangle =0,
\]
we get
\begin{align*}
\int_{[a_0\ldots a_k]} & f\d g^1 \wedge \d g^2  = \\
& \sum_{i=1}^m \left(\int_{[b\alpha_i^2\alpha_i^3]}f\d g^1 \wedge \d g^2 - \int_{[b\alpha_i^1\alpha_i^3]} f\d g^1 \wedge \d g^2
 + \int_{[b \alpha_i^1\alpha_i^2]} f\d g^1 \wedge \d g^2 \right)
= \\
& \sum_{i=1}^m \left(\int_{[\alpha_i^1\alpha_i^2 b]}f\d g^1 \wedge \d g^2 + \int_{[\alpha_i^2\alpha_i^3 b]}f\d g^1 \wedge \d g^2
+ \int_{[\alpha_i^3\alpha_i^1 b]} f\d g^1 \wedge \d g^2 
\right),
\end{align*}
the latter equality being due to the alternating property of the integral. Every one-dimensional edge
$[pq]$ of the triangulation not belonging to the boundary of $\Omega$ belongs to exactly two simplices of the triangulation
leading to two terms in the right-hand side of the latter equality, $\int_{[p q b]}f\d g^1 \wedge \d g^2$ and
$\int_{[q p b]}f\d g^1 \wedge \d g^2$ which cancel out due to the alternating property of the integral. Therefore, 
the right-hand side of the latter equality contains only terms of the type $\int_{[p q b]}f\d g^1 \wedge \d g^2$
with $[pq]$ belonging to the boundary of $\Omega$; due to the additivity property of the integral they all
sum up to the right-hand side of~\eqref{eq_delPoly2}. The rest of the statement follows now immediately from~\eqref{eq_delPoly2}
together with the respective properties of the integral over simplices.
\end{proof}

If $\Omega$ is a finite union of disjoint simple oriented polygons $\Omega_1,\ldots,\Omega_l$ then it is natural
to set
\begin{equation}\label{eq_delPoly3}
 \int_{\Omega} f\d g^1 \wedge \d g^2 := \sum_{i=1}^l \int_{\Omega_i} f\d g^1 \wedge \d g^2, 
\end{equation}
so that the above integral clearly exists under the conditions of Theorem~\ref{th_exist_discrInt1}.

Finally, we able to define naturally the $\int_{\Omega} f\d g^1 \wedge \d g^2$ for quite 
general bounded open sets $\Omega\subset \R^2$ with a chosen orientation.
To this aim for every $k\in\N$ let $P_k$ be the union of open squares with vertices in $2^{-k}\Z^2$ contained in $\Omega$.
Clearly this is a bounded open set which is a finite union of simple polygons. We assume all $P_k$ to be oriented in the same
way as $\Omega$.
The following result holds true.

\begin{theorem}\label{th_intOm1}
	Under conditions of Theorem~\ref{th_exist_discrInt1}, if additionally
$\Omega\subset \R^2$ is a bounded open set satisfying 
\begin{equation}\label{eq_boxbet1}
	\overline{\mathrm{dim}}_{\mathrm{box}} \partial \Omega < \beta_1+\beta_2,
\end{equation}	
	where $\overline{\mathrm{dim}}_{\mathrm{box}}$ stands for the upper box-counting dimension,
	there is the limit
\begin{equation}\label{eq_intlimOm1}
	\int_{\Omega} f\d g^1 \wedge \d g^2 := \lim_k \int_{P_k} f\d g^1 \wedge \d g^2.
\end{equation}	
	In this case 
	the map
	\[
	(f,g^1, g^2) \mapsto \int_{\Omega} f\d g^1 \wedge \d g^2
	\]	
	is a continuous multilinear functional over $C^\alpha\times C^{\beta_1}\times C^{\beta_2}$.
\end{theorem}

\begin{proof}
	Take a $d\in (\overline{\mathrm{dim}}_{\mathrm{box}} \partial \Omega, \beta_1+\beta_2)$. The set $P_{k+m}\setminus P_k$ can be
	naturally covered by triangles by dividing along the diagonal each of the squares of sidelength  $2^{- (k+m)}$ 
	with disjoint interiors  composing it. The total number of such squares is estimated from above by the number of squares
	with vertices in $2^{-k}\Z^2$ touching $\partial\Omega$, hence by $C (2^k)^d$ where $C>0$ depends only on $\partial\Omega$. Hence the number of triangles in the chosen cover of $P_{k+m}\setminus P_k$ is estimated by $2C (2^k)^d (2^m)^2$.
	Each triangle $\Delta$ in this cover has diameter $D:=2^{- (k+m)}$, and therefore by~\eqref{eq_estlimom1} together with~\eqref{eq_estom1a} one has
	\[
	\left|\int_{\Delta} f\d g^1(x)\wedge \d g^2(x)\right|\leq C' D^{\beta_1+\beta_2},
	\]
	where $C'>0$ depends only on $\|f\|_\alpha$, $[g^1]_{\beta_1}$, $[g^2]_{\beta_2}$. Thus 
	\begin{align*}
	\left|\int_{P_{k+m}} f\d g^1(x)\wedge \d g^2(x) - \int_{P_{k}} f\d g^1(x)\wedge \d g^2(x)\right| &\leq 2C(2^k)^d (2^m)^2 
	C' 2^{-(k+m)(\beta_1+\beta_2)}\\
	&\to 0\quad\mbox{as $k\to+\infty$}
	\end{align*}
	(even uniformly over bounded sets of $C^\alpha\times C^{\beta_1}\times C^{\beta_2}$) because of the assumption $\beta_1+\beta_2>d$. This shows that the sequence of integrals $\{\int_{P_k} f\d g^1 \wedge \d g^2\}_k$ is Cauchy, and hence the existence of the limit as claimed. This limit is clearly multilinear on $(f, g^1, g^2)$ since so is the integral over simple polygons, and its continuity
	over $C^\alpha\times C^{\beta_1}\times C^{\beta_2}$ follows  from that of the integral over polygons and of the fact that the above convergence is uniform over bounded sets of $C^\alpha\times C^{\beta_1}\times C^{\beta_2}$. 
\end{proof}

\begin{remark}
Clearly under the condition~\eqref{eq_boxbet1} the integral $\int_{\Omega} f\d g^1 \wedge \d g^2$ coincides with
the classical one if $f$, $g^1$ and $g^2$ are smooth.
\end{remark}

\begin{remark}
Combining Theorem~\ref{th_arbpart2} and Proposition~\ref{prop_intPoly1}, we have
that the integral $\int_{\Omega} f\d g^1 \wedge \d g^2$ in Theorem~\ref{th_intOm1}
may be also approximated directly by sums of either Stratonovich or It\^{o} germs over sufficiently fine %triangular partitions 
triangulations of $P_k$ (for sufficiently large $k$). 
\end{remark}

\begin{remark}
	If in the construction used in Theorem~\ref{th_intOm1} one substitutes the dyadic grids $2^{-k}\Z^2$ with some other ones
	(e.g. rotated and/or with sidelength of the cubes converging to zero with different speed), one would
	obtain under conditions of Theorem~\ref{th_intOm1} in exactly the same way the existence of the limit in~\eqref{eq_intlimOm1} (but now with different meaning of $P_k$), and its continuity and multilinearity 
	over $C^\alpha\times C^{\beta_1}\times C^{\beta_2}$. Since this limit for smooth $f$, $g^1$ and $g^2$ still coincides with the classical integral, we get therefore that it also coincides with $\int_{\Omega} f\d g^1 \wedge \d g^2$ over the whole $C^\alpha\times C^{\beta_1}\times C^{\beta_2}$, and hence the role of the particular sequence of grids in the definition~\eqref{eq_intlimOm1} is not essential.
\end{remark}

\section{Stratonovich type integrals of more irregular forms}\label{sec_strat_irreg1}
 
We consider in this section the integrals of the type
\[
\int_{\Omega} F(x, g(x))\d g^1(x)\wedge \d g^2(x)
\]
defined for H\"{o}lder functions
$g:=(g^1, g^2)\colon \R^2\to\R^2$ when $F\colon
\R^2\times\R^2\to \R$. In fact, it happens that if one uses a
Stratonovich-type construction, i.e.\ employs
alternating germs $\strat_{pqr}$
defined for $f(x):=F(x, g(x))$, then the above integral may be defined under much less restrictive requirements than those given by
Theorem~\ref{th_exist_discrInt1}. In particular, we are able to trade regularity of $g$ for the higher regularity of $F$.
Here we only limt ourselves to the case when the domain of integration $\Omega\subset\R^2$ is an oriented simplex (i.e.\ triangle $[pqr]$), since the case of more general domains can be easily treated as in section~\ref{sec_intOm1}. 

\begin{theorem}\label{th_irregInt1a}
Let $F\colon \R^2\times\R^2 \to \R$ such that
\begin{itemize}
\item[(i)] $u\mapsto F(u, \cdot) \in C(\R^2;C^{1,\gamma}(\R^2))$, $\gamma\in (0,1]$,
\item[(ii)]
$u\mapsto F(\cdot, u) \in C(\R^2; C^{\alpha})$,
\end{itemize}
 and let
$f(x):= F(x, g(x))$, where $g(x):=(g^1(x), g^2(x))$.
If $\beta_1+\beta_2>1$ and
\begin{equation}\label{eq_gambet12}
\begin{aligned}
\alpha+\beta_1+\beta_2 >2,\\
(1+\gamma)\beta_i+\beta_1+\beta_2>2, \quad i=1,2,
\end{aligned}
\end{equation}
then, with the notation of~\eqref{def_omnSn} the limit
\[
V_{pqr}(g):=\lim_{n \to \infty} \strat^n_{pqr}
\]
exists. Moreover, it is continuous and alternating as a function of $[pqr]$
fixed $g^1$ and $g^2$,
nonatomic in the sense that
\[
V_{pqr}(g)=0 \quad \mbox{when $|[pqr]|=0$},
\]
and continuous as the functional of $g$, so that it is reasonable to denote
\[
\int_{[pqr]} F(x, g(x))\d g^1(x)\wedge \d g^2(x) := V_{pqr}(g).
\]
\end{theorem}

\begin{remark}\label{rem_strat0a}
It is worth observing that~\eqref{eq_gambet12} implies $\beta_i > 1/3$, $i=1,2$.
In fact, assuming without loss of generality $\beta_1 <\beta_2$, we get from~\eqref{eq_gambet12}
%\[
 $(2+\gamma)\beta_1 +\beta_2 >2$,
%\]
and hence
\[
\beta_1> \frac{2-\beta_2}{2+\gamma}\geq \frac 1 3.
\]
On the other hand, $\beta_i > 1/2$, $i=1,2$, is clearly sufficient for the second inequality in~\eqref{eq_gambet12} to hold.
Note also that if $\beta_1=\beta_2=\beta$, and $F(x,y):= F(y)$ for every
$(x,y)\in \R^2\times\R^2$,
 then the first inequality of~\eqref{eq_gambet12}
is automatically satisfied since we may take $\alpha$ to be arbitrarily close
to $1$, and therefore~\eqref{eq_gambet12}
is equivalent to
$\beta> 2/(3+\gamma)$ (e.g. $\beta>1/2$ when $F\in C^{1,1}$), which is far less restrictive than what is asserted in Theorem~\ref{th_exist_discrInt1} (the latter requires in this case $\beta> 2/3$, since $f\in C^\beta$).
\end{remark}

\begin{remark}
It follows from the proof that the limit germ
\[
V_{pqr}:=\int_{[pqr]} F(x, g(x))\d g^1(x)\wedge \d g^2(x))
\]
 is continuous also with respect to $F$ (with respect to a topology compatible with~(i) and~(ii)).
\end{remark}

\begin{remark}\label{rem_Strat1D}
We notice that an analogous result is easy to obtain  in the one-dimensional case. Namely, roughly speaking, if $g \in C^{\beta}(\R)$ is H\"older continuous and $F\colon \R \times \R \to \R$ is $C^\alpha(\R)$ in the first variable and $C^{1, \gamma}(\R)$ in the second one, then the Stratonovich-type sums
\[ \sum_{i}  \frac1 2 \bra{ F(x_i, g(x_i) ) + F(x_{i+1}, g(x_{i+1}) ) } 
%\bra{g(x_{i+1}) - g(x_i) }
(\delta g)_{x_i x_{i+1}}
\]
over a sequence of partitions $(x_i)_{i}$ of $[a,b]$ converge as $\sup_{i} |x_{i+1} - x_{i} |  \to 0$
when
\begin{equation}\label{eq_abgam1}
\alpha+\beta>1\mbox{ and }\beta(2 +\gamma) >1.
\end{equation}
This can be deduced at once starting from the calculation
\[ \delta \theta_{pqr} =\frac 1 2 \det \bra{ \begin{array}{ll} \delta f_{p q} & \delta f_{pr} \\
\delta g_{pq} & \delta g_{pr} \end{array}}\]
 with $f_p := F(p, g_p)$ and $\theta_{pq} := \frac1 2 \bra{ f_p + f_q} \delta g_{pq}$. The assumptions on $f$ give the Taylor expansion
 \[ \delta f_{pq} = a_p \delta g_{pq} + O( |q-p|^{\alpha} + |q-p|^{\beta(1+\gamma)} )\]
 so that a cancelation occurs in the determinant providing $|\delta \theta_{pqr}| = O(  |q-p|^{\alpha+\beta} + |q-p|^{\beta(2+\gamma)})$, which gives the possibility to apply the 
 one-dimensional sewing lemma~\cite[lemma 2.1]{feyel_curvilinear_2006} 
 if~\eqref{eq_abgam1} holds. In particular, we notice that if $\alpha = \gamma = 1$, then $\beta>1/3$ is allowed, which is well below the threshold of H\"{o}lder exponents for the existence the Young integral (defined
 for $\beta>1/2$).  It is worth emphasizing that this is the peculiar feature of the Stratonovich integral, not of the It\^{o} one.
  In fact, if we take just $F(x, y):=y$, then the integral reduces to $\int_{[pq]}g\d g$, and for $g\in C^\beta(\R)$ with
  $\beta\in (1/3, 1/2]$ it is a limit of the sum of Stratonovich germs but in general not of It\^{o} germs. This is the case for instance when $g$ has infinite total quadratic variation, because the difference
  between the two germs over $[pq]$ is $(\delta g)_{pq}^2/2$, so that if the integral existed as the limit of sums
  of either of the germs, then the total quadratic variation of $g$ had to be finite.
\end{remark}

\begin{proof}
Let $f_u(t):= F(u,x+t(y-x))$ for $\{u,x,y\}\in \R^2$.
Writing
\begin{align*}
F(u,y)& =f_u(1) \\
      & = f_u(0) +\int_0^1 (f_u)'(s)\, ds
= f_u(0)+(f_u)'(0) +\int_0^1 ((f_u)'(s)-(f_u)'(0))\, ds\\
&= F(u,x) + \nabla F (u,\cdot)(x) \cdot (y-x) \\
& \qquad + \int_0^1 \left(\nabla
F(u,\cdot)(x+s(y-x))-\nabla F(u,\cdot)(x)\right)\cdot(y-x)\, ds,
\end{align*}
we get with $x:=g_u$, $y:=g_v$
the relationship
\begin{equation}\label{eq_deltFuv1}
\begin{aligned}
(\delta F)_{uv}& =
(\delta F(\cdot, g_v))_{uv} +(\delta F(u, \cdot))_{g_u g_v}\\
& = (\delta F(\cdot, g_v)_{uv} +
\delta g^1_{uv}\partial_1 F(u,\cdot)( g^1_u, g^2_u)  +
\delta g^2_{uv}\partial_2 F(u, \cdot)( g^1_u, g^2_u)  + R_{uv}, \\
& \qquad\qquad\qquad\mbox{ where}\\
R_{uv} &:=\delta g^1_{uv}
\int_0^1 \left(\partial_1 F (u, \cdot) (g^1_u + s\delta g^1_{uv}, g^2_u+ s\delta g^2_{uv})
- \partial_1 F(u, \cdot) (g^1_u, g^2_u)\right)\, ds \\
& \qquad +
\delta g^2_{uv}
\int_0^1 \left(\partial_2 F (u, \cdot) (g^1_u + s\delta g^1_{uv}, g^2_u+ s\delta g^2_{uv})
- \partial_2 F(u, \cdot)(g^1_u, g^2_u)\right)\, ds,
\end{aligned}
\end{equation}
so that
\begin{align*}
|(\delta F(\cdot, g_v)_{uv}|& \leq C|v-u|^\alpha,\\
% |R_{uv}|& \leq C(|\delta g^1_{uv} |^{1+\gamma} +  |\delta g^2_{uv} |^{1+\gamma})
 |R_{uv}|& \leq C \left(|\delta g^1_{uv}| + |\delta g^2_{uv}|\right))  \left((\delta g^1_{uv})^2 + (\delta g^2_{uv})^2\right)^{\gamma/2}
\end{align*}
for $(u,v)$ in a bounded set (the constant $C>0$ depending on this set).
From Lemma~\ref{lm_StratDet1} one gets therefore
\begin{equation}\label{eq_dstrat0a}
\begin{aligned}
(\delta\strat)_{pqrs}   = &
\frac 1 6\det \bra{ \begin{array}{lll}
(\delta F(\cdot, g_q))_{p q} & (\delta F(\cdot, g_r))_{pr} & (\delta F(\cdot, g_s ))_{ps} \\
\delta g^1_{pq} & \delta g^1_{pr} &\delta g^1_{ps} \\
\delta g^2_{pq} & \delta g^2_{pr} &\delta g^2_{ps}
\end{array}} +
\\
&
\frac 1 6\det \bra{ \begin{array}{lll}
R_{p q} & R_{pr} & R_{ps} \\
\delta g^1_{pq} & \delta g^1_{pr} &\delta g^1_{ps} \\
\delta g^2_{pq} & \delta g^2_{pr} &\delta g^2_{ps}
\end{array}},
\end{aligned}
\end{equation}
and hence
\begin{equation}\label{eq_ddeltom0}
\begin{aligned}
|(\delta\strat)_{pqrs}| &\leq
C
%[\delta g^1]_{\beta_1} [\delta g^2]_{\beta_2}
\left(
\diam([pqrs])^{\alpha+\beta_1+\beta_2} +
\diam([pqrs])^{(1+\gamma)(\beta_1\wedge \beta_2)+\beta_1+\beta_2}\right)\\
&\leq
C %[\delta g^1]_{\beta_1} [\delta g^2]_{\beta_2}
\diam([pqrs])^d
\end{aligned}
\end{equation}
with $d:=(\alpha\wedge (1+\gamma)(\beta_1\wedge \beta_2))+\beta_1+\beta_2$ and
$C>0$ depending continuously on $F$ (with respect to the topology compatible
with~(i) and~(ii)) and on
$[\delta g^i]_{\beta_i}$, $i=1,2$.
Recalling~\eqref{eq_estom1a} from Lemma~\ref{lm_estomgam1},
and that $\strat$ is alternating by the same Lemma,
while $d>2$
because of~\eqref{eq_gambet12},
we have that
Lemma~\ref{lm_Vnlimcont1a} applies with
\begin{align*}
\gamma_1  :=\beta_1+\beta_2>1, \quad C_1  &:=\norm{f}_\infty [\delta g^1]_{\beta_1} [\delta g^2]_{\beta_2},\\
\gamma_2  :=d>2, \quad C_2  &:=C,  %[\delta g^1]_{\beta_1} [\delta g^2]_{\beta_2},
\end{align*}
yielding the existence of continuous alternating germs
\begin{align*}
 S_{pq}& :=  \lim_{n \to \infty} S^n_{pq},\\
V_{pqr} & := \lim_{n \to \infty} \strat^n_{pqr} = \lim_{n \to \infty} (\strat^n _{pqr}- \delta S^n_{pqr}) + \delta S^n_{pqr}.
\end{align*}

It remains now to prove that fixed $[pqr]$, the map
\[
g\in C^{\beta_1} \times C^{\beta_2} \mapsto V_{pqr}(g)
\]
is continuous. To this aim let $\{g_k\}\subset C^{\beta_1} \times C^{\beta_2}$,
converging to $g$ pointwise as $k\to \infty$, and $[\delta g^1_k]_{\beta_1}+[\delta g^2_k]_{\beta_2} <C<+\infty$ for all $k\in \N$.
Let $f_k$, $R_k$, $S_k^n$, $S_k$, $\strat_k^n$, $\strat_k$, $V_k$ be the same as $f$, $R$,
$S^n$, $S$,
 $\strat^n$, $\strat$, $V$  respectively but with
$g^1_k$, $g^2_k$ instead of $g^1$, $g^2$.
Clearly, as in~\eqref{eq_ddeltom0} we have
\begin{equation}\label{eq_ddeltomk}
\begin{aligned}
|(\delta\strat_k)_{pqrs}|\leq C
\diam([pqrs])^d.
\end{aligned}
\end{equation}
The claim follows
now by Lemma~\ref{lm_Vnlimstab1} with $\gamma_2:=d$,
$\gamma_1=\beta_1+\beta_2$
(in fact,~\eqref{eq_omgam2} is given by~\eqref{eq_ddeltomk},
and~\eqref{eq_omgam1} is just~\eqref{eq_estom1a} from Lemma~\ref{lm_estomgam1}).
\end{proof}

\begin{remark}\label{rem_strat0b}
One could strengthen the above Theorem~\ref{th_irregInt1a} by proving the existence and continuity with respect to the data
of a more general Stratonovich type integral
\[
\int_{[pqr]} F(x, h(x))\d g^1(x)\wedge \d g^2(x),
\]
where $F$ is as in Theorem~\ref{th_irregInt1a},
$\psi\in C^{2,\gamma}(\R^2;\R^2)$, $\gamma\in (0,1]$,
$h^i\in C^{\beta_i}(\R^2)$,
$g^i:=\psi^i\circ h$, $i=1,2$ with $h:= (h^1, h^2)$, $\psi:= (\psi^1,\psi^2)$ and $\beta_i >1/2$, $i=1,2$ and satisfy
the first inequality of~\eqref{eq_gambet12}. In fact, letting $f(x):= F(x,h(x))$,
and using the notation of~\eqref{def_omnSn} we would have the existence of the limit
\[
\lim_{n \to \infty} \strat^n_{pqr} =: \int_{[pqr]} F(x, h(x))\d g^1(x)\wedge \d g^2(x).
\]
To show this, we adapt the arguments of the proof of the above Theorem~\ref{th_irregInt1a}, changing~\eqref{eq_dstrat0a}
with
\begin{equation}\label{eq_dstrat0b}
\begin{aligned}
(\delta\strat)_{pqrs}   = &
\frac 1 6\det \bra{ \begin{array}{lll}
(\delta F(\cdot, g_q))_{p q} & (\delta F(\cdot, g_r))_{pr} & (\delta F(\cdot, g_s ))_{ps} \\
\delta g^1_{pq} & \delta g^1_{pr} &\delta g^1_{ps} \\
\delta g^2_{pq} & \delta g^2_{pr} &\delta g^2_{ps}
\end{array}} +
\\
&
\frac 1 6\det \bra{ \begin{array}{lll}
R_{p q} & R_{pr} & R_{ps} \\
\delta g^1_{pq} & \delta g^1_{pr} &\delta g^1_{ps} \\
\delta g^2_{pq} & \delta g^2_{pr} &\delta g^2_{ps}
\end{array}} +
\\
&
\frac 1 6\det \bra{ \begin{array}{lll}
\nabla_h F(p,h_p)\cdot \delta h_{pq}  & \nabla_h F(p,h_p)\cdot \delta h_{pr} & \nabla_h F(p,h_p)\cdot \delta h_{ps} \\
r^1_{pq} &  r^1_{pr} & r^1_{ps} \\
\nabla\psi^2_{h_p}\cdot \delta h_{pq} & \nabla\psi^2_{h_p}\cdot \delta h_{pr} & \nabla\psi^2_{h_p}\cdot \delta h_{ps}
\end{array}} +
\\
&
\frac 1 6\det \bra{ \begin{array}{lll}
\nabla_h F(p,h_p)\cdot \delta h_{pq}  & \nabla_h F(p,h_p)\cdot \delta h_{pr} & \nabla_h F(p,h_p)\cdot \delta h_{ps}  \\
\nabla\psi^1_{h_p}\cdot \delta h_{pq} & \nabla\psi^1_{h_p}\cdot \delta h_{pr} & \nabla\psi^1_{h_p}\cdot \delta h_{ps}
 \\
r^2_{pq} &  r^2_{pr} & r^2_{ps}
\end{array}}
,
\end{aligned}
\end{equation}
where
\[
r^i_{uv} := \delta g^i_{uv} -(\nabla\psi_i)_{h_u}\cdot \delta h_{uv}, \quad i=1,2.
\]
Then the first two terms in~\eqref{eq_dstrat0b} are estimated by $C\diam([pqrs])^{d_1}$ with $d_1>2$ as in~\eqref{eq_ddeltom0}
because of~\eqref{eq_gambet12} (the second inequality of which is automatically satisfied in view of Remark~\ref{rem_strat0a}
due to the requirement $\beta_i >1/2$, $i=1,2$), while the other two are estimated by
$C\diam([pqrs])^{d_2}$ with $d_2:= 4(\beta_1\wedge \beta_2)>2$, because
\[
|r^i_{uv}|\leq C |u-v|^{2(\beta_1\wedge\beta_2)},
\]
and thus $|\delta\strat_{pqrs}|\leq C\diam([pqrs])^{d}$,
the constants in all the above estimates depending continuously on the data. This allows to proceed as in the proof of Theorem~\ref{th_irregInt1a} showing the existence and continuity with respect to the data of the above integral.
\end{remark}

\begin{proposition}[chain rule]\label{prop_irregIntChrule2}
Let $F$ be as in Theorem~\ref{th_irregInt1a},
$\psi\in C^{2,\gamma}(\R^2;\R^2)$, $\gamma\in (0,1]$,
$h^i\in C^{\beta_i}(\R^2)$,
and $g^i:=\psi^i\circ h$, $i=1,2$, where $h:= (h^1, h^2)$, $\psi:= (\psi^1, \psi^2)$.
If $\beta_i>1/2$, $i=1,2$ and the first inequality of~\eqref{eq_gambet12} holds,
then
\begin{equation}\label{eq_StratChainrule1}
\begin{aligned}
\int_{[pqr]} & F(x, h(x))  \d g^1(x)\wedge \d g^2(x) \\
& =
\int_{[pqr]} F(x, h(x)) \det D\psi(h^1(x), h^2(x))\d h^1(x)\wedge \d h^2(x)%-\Delta(F, g^1, g^2).
\end{aligned}
\end{equation}
\end{proposition}

Note that the integral on the right-hand side of~\eqref{eq_StratChainrule1}
exists,
is continuous and alternating as a function of $[pqr]$
fixed $h^1$ and $h^2$,
and continuous as the functional of $h^1$, $h^2$ by Theorem~\ref{th_irregInt1a}.

\begin{proof}
The equality~\eqref{eq_StratChainrule1} is true when $g^i$ are smooth. The general case follows from continuity of
the integrals on the left and righthand sides of~\eqref{eq_StratChainrule1} with respect to the pointwise
convergence of $g^i$, $i=1,2$ with uniformly bounded H\"{o}lder constants.
\end{proof}

We may give an interpretation of the above results in the spirit of theorem~3.2 from~\cite{AlbertiMajer94}.
Namely, a smooth (say, $C^1$)
function $g=(g_1,g_2)\colon [pqr]\subset \R^2\to \R^2$
can be naturally identified with the smooth surface representing its graph, and therefore, with the De Rham $2$-current $T_g$ over $[pqr]\times\R^2$ (endowed with orthogonal coordinates $(x,y):=(x^1, x^2, y^1, y^2)$) defined by
\begin{eqnarray}%\label{eq_def_Tgcurr1}
%\begin{aligned}
\label{eq_def_Tgcurr1}
T_g(F\d x^1\wedge \d x^2) &:=\displaystyle\int_{[pqr]} F(x,g(x))\,\d x^1\wedge \d x^2,\\
\label{eq_def_Tgcurr2}
T_g(F\d x^i\wedge \d y^j) &:=\displaystyle\int_{[pqr]} F(x,g(x))\, \d x^i\wedge \d g^j(x),\\
\label{eq_def_Tgcurr3}
T_g(F\d y^1\wedge \d y^2) &:=\displaystyle\int_{[pqr]} F(x,g(x))\, \d g^1(x)\wedge \d g^2(x),
%\end{aligned}
\end{eqnarray}
for every $f\in C^2([pqr]\times\R^2)$.

\begin{proposition}\label{prop_irregIntCurr1}
If $g^i\in C^{\beta_i}$, $i=1,2$, %then under conditions~\eqref{eq_def_Tgcurr4}
with
\begin{equation}\label{eq_def_Tgcurr4}
 %\beta_1+\beta_2>1,
 3\beta_1+\beta_2>2,\quad  3\beta_2+\beta_1>2,
\end{equation}
then the map $g\mapsto T_g$ between $C^1([pqr];\R^2)$
and the space $D_2([pqr]\times\R^2)$ of $2$-currents in $[pqr]\times\R^2$
endowed with its weak (pointwise) topology
admits the unique continuous extension  to the space
$C^{\beta_1}\times C^{\beta_2}$ (the continuity being intended,
as usual, with respect to pointwise convergence with uniformly bounded H\"{o}lder constants).
\end{proposition}

\begin{proof}
If $g^i\in C^{\beta_i}$, $i=1,2$, then the formulae~\eqref{eq_def_Tgcurr1},~\eqref{eq_def_Tgcurr2} and~\eqref{eq_def_Tgcurr3} still make sense
for an  $F\in C^2([pqr]\times\R^2)$ if one interprets the integrals involved in the sense of Stratonovich.
Namely, one defines the integral
\begin{itemize}
\item[(A)] in~\eqref{eq_def_Tgcurr1}, say, in the usual Riemann (or Lebesgue) sense (which in this case is equivalent to the Stratonovich integral),
\item[(B)]
in~\eqref{eq_def_Tgcurr3} in the sense of
Theorem~\ref{th_irregInt1a} (with $\alpha:=1$, $\gamma:=1$), and
\item[(C)]
in~\eqref{eq_def_Tgcurr2} again in the sense
of Theorem~\ref{th_irregInt1a} but with $x^i$ in place of $g^1$, $g^j$ in place of $g^2$, and $\bar F$
in place of $F$, where $\bar F$ is defined by
\[
\bar F(x^1, x^2, y^1, y^2):=
\left\{
\begin{array}{rl}
F(x^1, x^2, g^1(x), y^2), & i=1, j=2,\\
F(x^1, x^2, y^1, g^2(x)), & i=2, j=1,
\end{array}
\right.
\]
and with $\gamma:=1$, $\alpha:=\beta_1$ and $1$ in place of $\beta_1$ for the case $i=1$, $j=2$ or $\alpha:=\beta_2$ and $1$ in place of $\beta_2$
for the case $i=2$, $j=1$.
\end{itemize}
Note that~\eqref{eq_def_Tgcurr4} makes
Theorem~\ref{th_irregInt1a} to be applicable with such data.

Continuity of the map $g\mapsto T_g$ between $C^{\beta_1}\times C^{\beta_2}$ and the space of currents endowed with its weak (pointwise) topology is given
by Theorem~\ref{th_irregInt1a}. The fact that it is the unique continuous extension of its restriction to $C^1\times C^1$ follows from the density of $C^1$
in any H\"{o}lder space (with respect to the uniform convergence with bounded
H\"{o}lder constants).
\end{proof}

\begin{remark}
The proof of Proposition~\ref{prop_irregIntCurr1} shows that
the formulae~\eqref{eq_def_Tgcurr1},~\eqref{eq_def_Tgcurr2}
and~\eqref{eq_def_Tgcurr3} still make sense for the current $T_g$ with
$g\in C^{\beta_1}\times C^{\beta_2}$
 when
$F\in C^2([pqr]\times\R^2)$ (in fact, even for $F\in C^{1,1}$), if  one interprets the integrals appearing there in the sense
of Stratonovich, i.e.\ as in Theorem~\ref{th_irregInt1a}
(in particular, in~\eqref{eq_def_Tgcurr1} it may be interpreted as the usual
Riemann or Lebesgue integral).
\end{remark}

\begin{remark}
Theorem~3.2 from~\cite{AlbertiMajer94} says that the map $g\mapsto T_g$
defined by the formulae~\eqref{eq_def_Tgcurr1},~\eqref{eq_def_Tgcurr2}
and~\eqref{eq_def_Tgcurr3} between
$C^1([pqr];\R^2)$ and the space of currents endowed with its weak topology
admits a unique continuous extension to the Sobolev space
$W^{1,1}_{loc}([pqr];\R^2)$ (even sequentially weakly continuous one). It is worth noting that the extended current
may be then defined for continuous differential forms
(i.e.\ with $F$ just continuous), while here we have to require
that the forms be smoother (in fact, requesting $F$ to be $C^2$, we are guaranteed only that the extended current $T_g$ be defined over twice continuously differential forms). One may weaken the regularity requirement for forms
(e.g. requesting that $F$ might be less regular than $C^2$), but this will inevitably strengthen the requirement of~\eqref{eq_def_Tgcurr4} on the regularity of $T_g$.
\end{remark}

\begin{remark}\label{rm_discr_deg1}
In order to identify the extension with the ``second order Riemann-Stieltjes'' integral introduced in \cite{zust_integration_2011}, we extend by continuity the identity
\begin{equation} \label{eq_degree} \int_{\R^2} f(x) \deg \bra{ (h^1, h^2), [pqr], x} \d x = \int_{[pqr]} f(h^1, h^2) \d h^1 \wedge \d h^2\end{equation}
for every $f \in C^{1, \gamma}$ from smooth functions $(h^1, h^2)$ to $h_1 \in C^{\beta_1}$, $h_2 \in R^{\beta_2}$. %with \eqref{eq_resto_sim_1}.
In combination with \cite[theorem 4.3]{zust_integration_2011} this identifies the two integrals. Formula \eqref{eq_degree} follows by continuity and approximation.

We also notice that continuity of the right hand side in \eqref{eq_degree} gives the following quantitative continuity of degree of H\"older maps:
\[ \int_{\R^2} f(x) \bra{\deg\bra{ (h^1, h^2), [pqr], x}  - \deg\bra{ (k^1, k^2), [pqr], x} } \d x \le \norm{f}_{1,\gamma} \norm{h-k}_{\beta}\]
\end{remark}

\appendix

\section{Existence, uniqueness and stability of integrals}

In this section we assume that $\omega$ be an abstract
$2$-germ in $D\subset \R^2$
(i.e.\ not necessarily the one defined by~\eqref{eq_germzust2}
satisfying
\begin{eqnarray}
\label{eq_omgam1}
|\omega_{pqr}|\leq C_1 \diam([pqr])^{\gamma_1}, %\quad \gamma_1>1
\\
\label{eq_omgam2}
|(\delta\omega)_{pqrs}|\leq C_2 \diam([pqrs])^{\gamma_2}, %\quad \gamma_2>2
\end{eqnarray}
with %some $\gamma_1>1$, $\gamma_2>2$ and
positive constants $\gamma_1$, $\gamma_2$,
$C_1$, $C_2$ independent on $[pqr]$ and $[pqrs]$.
We define then $\omega^n$ and $S^n$ by %~\eqref{def_omnSn}.
\begin{equation}\label{def_omnSn_abstr0}
 \omega^n_{pqr} := \ang{ \dya^n [pqr], \omega}, \quad
%S^n_{pq} := \sum_{i=0}^{n} \ang{ \fill \cut^i [pq], \omega}.
S^n_{pq} := \sum_{i=0}^{n-1} \ang{ \fill \cut^i [pq], \omega}.
\end{equation}

We prove here the existence of limits $\lim_n \omega^n$ and $\lim_n S^n$
and their basic stability properties. Note that we do not prove here
that the respective germs are nonatomic and additive
(although in fact this could be proven), as it is usually done in the sewing lemma.

\begin{lemma}\label{lm_Vnlimcont1a}
Under the conditions~\eqref{eq_omgam1} and~\eqref{eq_omgam2}
if $\omega$ is alternating, then
\begin{eqnarray}
\label{eq_incrSn1a}
|S^{n+1}_{pq} - S^n_{pq}| \le C\diam([pq])^{\gamma_1} 2^{n(1- \gamma_1)},\\
\label{eq_incrVn1a}
 \begin{aligned} |\langle [pqr], (\omega^n - \delta S^n)
%&
- (\omega^{n+1} - \delta S^{n+1})\rangle|
%\\
 % &
\le C \diam([pqr])^{\gamma_2} 2^{n(2-\gamma_2)}
  \end{aligned}
\end{eqnarray}
with $C>0$.
In particular, if $\gamma_1>1$ and $\gamma_2>2$, then the germs
\begin{align*}
 S_{pq}& :=  \lim_{n \to \infty} S^n_{pq},\\
V_{pqr} & := \lim_{n \to \infty} \omega^n_{pqr} = \lim_{n \to \infty} (\omega^n _{pqr}- \delta S^n_{pqr}) + \delta S^n_{pqr}
\end{align*}
are well defined,
continuous (if so is $\omega$), alternating and
\begin{eqnarray}
\label{eq_estlimS1aa}
| S^n_{pq}-S_{pq}|\leq  C \diam([pq])^{\gamma_1} 2^{n(1- \gamma_1)},\\
\label{eq_estlimom1aa}
\begin{aligned}
|\omega^n_{pqr}-V_{pqr} %&
-\delta(S^n-S)_{pqr}| \leq
%\\
%&
C \diam([pqr])^{\gamma_2 } 2^{n(2-\gamma_2)},
\end{aligned}\\
\label{eq_estlimom1ba}
|\omega^n_{pqr}-V_{pqr}| \leq C \diam([pqr])^{\gamma_1\wedge \gamma_2} 2^{n(1- \gamma_1\wedge\gamma_2)}.
\end{eqnarray}
\end{lemma}

\begin{proof}
For the readers' convenience we organize the proof in several steps.

{\em Step 1}.
To prove~\eqref{eq_incrVn1a},
observe that for some geometric map $\rho\colon \simp^2(D)\to \chain^3(D)$ one has
\begin{equation}\label{eq_omS0aa}
\begin{aligned}
\omega^1_{p_0p_1p_2} & -\omega^0_{p_0p_1p_2}= \ang{\dya[p_0p_1 p_2],\omega}-\ang{[p_0p_1 p_2], \omega} \\
& =\ang{\partial \rho([p_0p_1 p_2]), \omega} -\ang{\fill[p_0p_1],\omega} + \ang{\fill[p_1 p_2],\omega} +
\ang{\fill[p_2p_0],\omega}\\
&= \ang{\rho([p_0p_1 p_2]), \delta\omega} +
\ang{\fill\partial [p_0p_1p_2],\omega}.
\end{aligned}
\end{equation}
Moreover,
\[
\rho([p_0p_1 p_2])=\sum_{i=1}^4 Q_i,\quad Q_i\in \simp^3(D), \diam Q_i\leq \diam([p_0p_1 p_2]), i=0,\ldots,2,
\]
and therefore by~\eqref{eq_omgam2} we have
\begin{equation}\label{eq_rhodom1aa}
|\ang{\rho([p_0p_1 p_2]), \delta\omega}|\leq C \diam([p_0p_1 p_2])^{\gamma_2},
\end{equation}
with $C:= 4 C_1$.
Writing then
$\dya [pqr]=\sum_{i=1}^{2^{2n}} \Delta_i$ with $\Delta_i\in \simp^2(D)$ being dyadic simplices equal up to translations to to $2^{-n}_{\#}[pqr]$, we get from~\eqref{eq_omS0aa}
\[
\ang{\Delta_i,\omega^1}-\ang{\Delta_i,\omega^0} = \ang{\rho(\Delta_i), \delta\omega} + \ang{\fill \partial \Delta_i,\omega},
\]
and summing the latter expressions over $i=1,\ldots, 2^{2n}$, we arrive at
\begin{equation}\label{eq_omS0ca}
\begin{aligned}
\omega^{n+1}_{pqr}-\omega^n_{pqr} &=\sum_{i=1}^{2^{2n}}\ang{\Delta_i,\omega^1-\omega^0} \\
&= \sum_{i=1}^{2^{2n}}\ang{\rho(\Delta_i), \delta\omega} + \ang{\fill\cut^n\partial [pqr],\omega},
\end{aligned}
\end{equation}
since if $\Delta_i$ and $\Delta_j$ have a common couple of vertices, say, $p_0$ and $p_1$, then by alternating property of $\omega$ one has
\[
\ang{\fill [p_0p_1],\omega}=- \ang{\fill [p_1p_0],\omega},
\]
i.e.\ the respective terms cancel out from the above sum, while the terms coming from the sides of dyadic simplices belonging to the boundary of $[pqr]$ remain, their sum giving rise to $\ang{\fill\cut^n\partial [pqr],\omega}$.
Observing that
\[
\ang{\fill\cut^n\partial [pqr],\omega}= \ang{[pqr],\delta S^{n+1}-\delta S^n}
\]
and
rewriting~\eqref{eq_omS0ca} with this help, we arrive at
\begin{equation}\label{eq_domn1na}
\begin{aligned}
(\omega^{n+1}_{pqr}-(\delta S^{n+1})_{pqr}) - (\omega^n_{pqr}-(\delta S^n)_{pqr}) &= \sum_{i=1}^{2^{2n}}\ang{\rho(\Delta_i), \delta\omega}.
\end{aligned}
\end{equation}
Therefore,
\begin{align*}
|(\omega^{n+1}_{pqr} &-(\delta S^{n+1})_{pqr}) - (\omega^n_{pqr}-(\delta S^n)_{pqr})| \leq \sum_{i=1}^{2^{2n}}
 |\ang{\rho(\Delta_i), \delta\omega}|\\
& \leq C\sum_{i=1}^{2^{2n}}\diam(\Delta_i)^{\gamma_2}\quad\mbox{by~\eqref{eq_rhodom1aa}}\\
& \leq C 2^{2n} \left(\frac{\diam([pqr])}{2^n}\right)^{\gamma_2}
\end{align*}
as claimed.

{\em Step 2}. The estimate~\eqref{eq_incrSn1a} follows with $C:=C_1$
just observing that
\[
S^{n+1}_{pq}-S^n_{pq} = \ang{\fill\cut^n[pq],\omega},
\]
while in view of~\eqref{eq_omgam1} and of the definition of $\fill\cut^n$
one has
\[
|\ang{\fill\cut^n[pq],\omega}|\leq C_1 2^n \left(\frac{\diam([pq]}{2^n}\right)^{\gamma_1}.
\]

{\em Step 3}.
Existence of $S$ and $V$ follow now from~\eqref{eq_incrSn1a} and~\eqref{eq_incrVn1a} respectively. Since $\omega$ is alternating, then so are $\omega^n$ and $S^n$, and therefore also $V$ and $S$.
Now, the continuity of $\omega$ implies that of $S^n$ and $\omega^n$ for each fixed $n\in \N$, and hence the continuity of $S$ and $V$ follow
from~\eqref{eq_estlimS1aa} and~\eqref{eq_estlimom1ba} respectively once they are proven. E.g.\ to prove continuity of $S$, for
$[pq]\subset D$ and $[rs]\subset D$ with $D$ bounded, given an $\varepsilon>0$,
we choose an $n\in \N$ such that $C\diam D 2^{n(1-\gamma_1)}< \varepsilon/3$, so that
\begin{align*}
|S_{pq}-S_{rs}| &\leq |S_{pq}-S^n_{pq}|+ |S^n_{pq} - S^n_{rs}| + |S^n_{rs}-S_{rs}|\\
              & \leq 2\varepsilon/3 +|S^n_{pq} - S^n_{rs}|\quad\mbox{by~\eqref{eq_estlimS1aa} and the choice of $\varepsilon$},
\end{align*}
so that it is enough to find a $\delta=\delta(n,\varepsilon)>0$ such that
$|S^n_{pq} - S^n_{rs}|<\varepsilon/3$ once $|p-q|+|r-s|<\delta$ to get
$|S_{pq}-S_{rs}|<\varepsilon$. The proof of continuity of $V$ is completely analogous (with the use of~\eqref{eq_estlimom1ba} instead of~\eqref{eq_estlimS1aa}).

{\em Step 3}. Finally, we prove~\eqref{eq_estlimS1aa},~\eqref{eq_estlimom1aa} and~\eqref{eq_estlimom1ba}.
The inequality~\eqref{eq_estlimS1a} is proven by the chain of estimates
\begin{align*}
| S^n_{pq}-S_{pq}| & = %\left| \sum_{k=2}^n (S^k_{pq}-S^{k-1}_{pq}) -\sum_{k=2}^\infty (S^k_{pq}-S^{k-1}_{pq})\right|\\
   %& =
   \left| \sum_{k=n+1}^\infty (S^k_{pq}-S^{k-1}_{pq}) \right|
%\\
   %&
\leq C \diam([pq])^{\gamma_1} \sum_{k=n+1}^\infty
     2^{k(1- \gamma_1)}\quad\mbox{by~\eqref{eq_incrSn1a}}\\
   &
   \leq C\frac{2^{n(1- \gamma_1)}}{1-2^{1- \gamma_1}}\diam([pq])^{\gamma_1}.
\end{align*}
Analogously,~\eqref{eq_estlimom1aa} follows from
\begin{align*}
|\omega^n_{pqr}-V_{pqr}  & -\delta(S^n-S)_{pqr}| = |(\omega^n_{pqr}-\delta S^n _{pqr})  -(V_{pqr}-\delta S_{pqr})|\\
& =   \left| \sum_{k=n+1}^\infty \left( (\omega^k_{pqr}-\delta S^k _{pqr})  -(\omega^{k-1}_{pqr}-\delta S^{k-1}_{pqr})\right) \right|\\
&\leq C \diam([pqr])^{\gamma_2 }
\sum_{k=n+1}^\infty 2^{k(2-\gamma_2)}
\quad\mbox{by~\eqref{eq_incrVn1a}}\\
&\leq C\frac{2^{n(2-\gamma_2)}}{1-2^{2-\gamma_2}} \diam([pqr])^{\gamma_2}.
\end{align*}
Finally,~\eqref{eq_estlimS1aa} gives
\[
| \delta(S^n-S)_{pqr}|\leq C\diam([pqr])^{\gamma_1} 2^{n(1- \gamma_1)},
\]
which together with~\eqref{eq_estlimom1aa} implies~\eqref{eq_estlimom1ba}
for $\diam([pqr])<1$ (which is enough since $D$ is assumed bounded),
thus concluding the proof.
\end{proof}

As a result of Lemma~\ref{lm_Vnlimcont1a} we have that $V$ and $S$ satisfy
\begin{eqnarray*}
\label{eq_estlimS1aa0}
|S_{pq}|\leq  C \diam([pq])^{\gamma_1},\\
\label{eq_estlimom1aa0}
\begin{aligned}
|\omega_{pqr}-(V
-\delta S)_{pqr}| \leq
C \diam([pqr])^{\gamma_2}.
\end{aligned}
\end{eqnarray*}
In particular, if $\gamma_1>1$ and $\gamma_2>2$ this implies
\begin{eqnarray}
\label{eq_estlimS1aa2}
|S_{pq}|\leq o(\diam([pq])) \quad\mbox{as $\diam([pq])\to 0$},\\
\label{eq_estlimom1aa2}
\begin{aligned}
|\omega_{pqr}-(V
-\delta S)_{pqr}| \leq
o(\diam([pqr])^2) \quad\mbox{as $\diam([pqr])\to 0$}.
\end{aligned}
\end{eqnarray}
Moreover, since
\[
S_{pq}=\sum_{i=0}^\infty \ang{ \fill \cut^i [pq], \omega},
\]
then one has
\begin{eqnarray}
\label{eq_Som2}
(\delta S)_{prq}= \omega_{prq} \quad\mbox{when $r=\frac{p+q}{2}$}.
\end{eqnarray}
Finally,
\begin{eqnarray}
\label{eq_Vdyainv1}
\ang{\dya [pqr], V} =\ang{[pqr], V}.
\end{eqnarray}

The following curious result,  though not used elsewhere in this paper,
gives the uniqueness of such a couple $(S,V)$ for a given $\omega$.

\begin{lemma}\label{lm_SVunique1}
%Let $\gamma_1>1$ and $\gamma_2>2$.
Given an %continuous
$\omega\in \germ^2(D)$, the couple of germs
$(S,V)\in \germ^1(D)\times \germ^2(D)$
satisfying~\eqref{eq_estlimS1aa2},~\eqref{eq_estlimom1aa2}, \eqref{eq_Som2}
and~\eqref{eq_Vdyainv1} is unique.
\end{lemma}

\begin{proof}
Suppose that there are two couples
$(S_i,V_i)\in \germ^1(D)\times \germ^2(D)$, $i=1,2$ satisfying~\eqref{eq_estlimS1aa2},~\eqref{eq_estlimom1aa2},~\eqref{eq_Som2}
and~\eqref{eq_Vdyainv1}. Then for $S:=S_1-S_2$ and $V:=V_1-V_2$ we get
\begin{eqnarray}
\label{eq_estlimS1aa3}
|S_{pq}|\leq o(\diam([pq]) \quad\mbox{as $\diam([pq])\to 0$},\\
\label{eq_estlimom1aa3}
\begin{aligned}
|(V
-\delta S)_{pqr}| =
o\left(\diam([pqr])^2\right), \quad\mbox{as $\diam([pqr])\to 0$, and}
\end{aligned}\\
\label{eq_Som3}
(\delta S)_{prq}= 0. \quad\mbox{when $r=\frac{p+q}{2}$}.
\end{eqnarray}
For each $n\in \N$ dividing dyadically the line segment $[pq]$
by consecutive points
\[
r_j:= \left(1-  \frac j {2^n}\right) p +  \frac j {2^n}q , \quad j=0, \ldots, 2^n,
\]
we get
\begin{align*}
S_{pq}=\sum_{j=0}^{2^n} S_{r_jr_{j+1}}
\end{align*}
by~\eqref{eq_Som3}, and hence,
\[
|S_{pq}|\leq \sum_{j=0}^{2^n} |S_{r_jr_{j+1}}|\leq  2^n
o\left(\frac{|pq|}{2^{n}}\right) = |pq| o\left(1\right)
\]
as $n\to 0$,
by~\eqref{eq_estlimS1aa3}, and taking the limit in the above inequality as $n\to \infty$, we get $S_{pq}=0$.
Then~\eqref{eq_estlimom1aa3} is reduced to
\begin{equation}\label{eq_estlimom1aa4}
|V_{pqr}| =
o\left(\diam([pqr])^2\right)\quad\mbox{as $\diam([pqr])\to 0$}.
\end{equation}
Recalling that $\ang{(\dya)^n[pqr],V}=\ang{[pqr], V}$ for every $n\in \N$
(because both $V_1$ and $V_2$ are assumed to satisfy~\eqref{eq_Vdyainv1}), we get
using~\eqref{eq_estlimom1aa4} the estimate
\begin{align*}
|V_{pqr}| =%|(\dya')^n V)_{pqr}|=
|\ang{\dya^n[pqr], V}| &= 2^{2n}
o\left(\frac{\diam([pqr])^2}{2^{2n}}\right) \\
& =\diam([pqr])^2 o(1) \to 0
\end{align*}
as $n\to \infty$. This implies $V=0$ concluding the proof.
\end{proof}

Consider now a sequence of continuous alternating germs $\{\omega_k\}\subset \germ^2(D)$ satisfying
\begin{eqnarray}
\label{eq_omkgam1}
|(\omega_k)_{pqr}|\leq C_1 \diam([pqr])^{\gamma_1}, %\quad \gamma_1>1
\\
\label{eq_omkgam2}
|(\delta\omega_k)_{pqrs}|\leq C_2 \diam([pqr])^{\gamma_2}, %\quad \gamma_2>2
\end{eqnarray}
with %some $\gamma_1>1$, $\gamma_2>2$ and
positive constants $\gamma_1>1$, $\gamma_2>2$,
$C_1$, $C_2$ independent on $[pqr]$, $[pqrs]$ and $k$.
\begin{equation}\label{def_omknSn}
 (\omega_k^n)_{pqr} := \ang{ \dya^n [pqr], \omega_k}, \quad
(S^n_k)_{pq} := \sum_{i=0}^{n-1} \ang{ \fill \cut^i [pq], \omega_k}.
\end{equation}
Lemma~\ref{lm_Vnlimcont1a} guarantees the existence for each $k\in \N$ of
continuous alternating germs
\begin{align*}
 (S_k)_{pq}& :=  \lim_{n \to \infty} (S_k^n)_{pq},\\
(V_k)_{pqr} & := \lim_{n \to \infty} (\omega_k^n)_{pqr} = \lim_{n \to \infty} ((\omega_k^n )_{pqr}- \delta (S_k^n)_{pqr}) + (\delta S^n_k)_{pqr}.
\end{align*}
Suppose further that $\omega_k\to \omega$ pointwise. Then clearly the latter
satisfy~\eqref{eq_omgam1} and~\eqref{eq_omgam1} and thus
Lemma~\ref{lm_Vnlimcont1a} provides the existence of continuous alternating germs
\begin{align*}
 S_{pq}& :=  \lim_{n \to \infty} S^n_{pq},\\
V_{pqr} & := \lim_{n \to \infty} \omega^n_{pqr} = \lim_{n \to \infty} (\omega^n _{pqr}- \delta S^n_{pqr}) + \delta S^n_{pqr},
\end{align*}
where $\omega^n$ and $S^n$ are defined by~\eqref{def_omnSn_abstr0}.
The following stability statement is valid.

\begin{lemma}\label{lm_Vnlimstab1}
Under the above conditions one has $S=\lim_k S_k$ and
$V=\lim_k V_k$
pointwise.
\end{lemma}

\begin{proof}
We note first that
\begin{align*}
|(S_k^n)_{pq}- S^n_{pq}| & = \abs{\ang{\fill\cut^n[pq], \omega-\omega_k}}
%\\
%&
\leq C_1 2^n \left(\frac{\diam ([pq])}{2^n}\right)^{\gamma_1}\to 0
\end{align*}
as $n\to \infty$ uniformly in $k$, which implies  $S=\lim_k S_k$ pointwise
via the standard estimate
\[
|(S_k)_{pq}- S_{pq}|\leq |(S_k)_{pq}- (S_k^n)_{pq}| + |(S_k^n)_{pq}- S^n_{pq}| +
 |S_{pq}- S^n_{pq}|.
\]

Writing
\begin{equation*}\label{eq_domn1nba0}
\begin{aligned}
(V_k-\delta S_k)- (V-\delta S) = -& \left(\omega_k^n-V_k -\delta(S_k^n-S_k)\right) +\\
& \left(\omega^n-V -\delta(S^n-S)\right) - \left(\omega^n-\omega_k^n
-\delta(S^n-S_k^n)\right),
\end{aligned}
\end{equation*}
and evaluating the latter relationship at $[pqr]$,  using
\begin{eqnarray*}
\begin{aligned}
|(\omega_k^n)_{pqr}-(V_k)_{pqr}
-\delta(S^n_k-S_k)_{pqr}| \leq
C 2^{n(2-\gamma_2)},\\
|\omega^n_{pqr}-V_{pqr}
-\delta(S^n-S)_{pqr}| \leq
C 2^{n(2-\gamma_2)}
\end{aligned}
\end{eqnarray*}
with $C>0$ independent of $n$ and $k$, we arrive at the estimate
\begin{equation}\label{eq_domn1nba1}
\begin{aligned}
|(V_k-\delta S_k)_{pqr}- (V-\delta S)_{pqr}| &\leq
2 C 2^{n(2-\gamma_2)} +
 \left|\omega^n_{pqr}-(\omega_k^n)_{pqr}
-\delta(S^n-S_k^n)_{pqr}\right|.
\end{aligned}
\end{equation}
Given an $\varepsilon>0$ we fix an $n=n(\varepsilon)\in \N$ such that
the first term on the right-hand side of~\eqref{eq_domn1nba1} does not exceed
$\varepsilon/2$, and since  $\lim_k S_k^n= S^n$ and
$\lim_k \omega_k^n= \omega^n$ pointwise, we get that also the second term on the does not exceed $\varepsilon/2$ for all sufficiently large $k$.
This means
\[
V-\delta S=\lim_k (V_k-\delta S_k)
\]
pointwise and therefore $V=\lim_k V_k$ pointwise since $\lim_k S_k=S$, concluding the proof.
\end{proof}

\section*{Acknowledgements} The authors are grateful to Elisabetta Chiodaroli and to the participants 
of the Workshop ``Rough calculus and weak geometric structures'' (Moscow, 2018) Valentino Magnani, Annalisa Massaccesi, Stefano Modena, Khadim War and Roger Z\"{u}st for stimulating discussions that largely influenced the research leading to this paper.

\bibliographystyle{plain}
%\bibliography{biblio-sewing-two}

\end{document}